\documentclass[10pt]{amsart}

\usepackage{amsfonts,amsmath,amsthm,amssymb,amscd,latexsym,enumerate,etoolbox,mathrsfs,comment, graphicx}
\usepackage[all]{xy}
\usepackage{thmtools, thm-restate}

\newcommand{\N}{\mathbb{N}}

\newcommand{\Z}{{\mathbb Z}}

\renewcommand{\phi}{\varphi}

\newcommand{\into}{\hookrightarrow}

\theoremstyle{plain}
    \newtheorem{theorem}{Theorem}[section]
    \newtheorem{lemma}[theorem]{Lemma}
    \newtheorem{corollary}[theorem]{Corollary}
    \newtheorem{proposition}[theorem]{Proposition}
    
\theoremstyle{definition}
    
    \newtheorem{example}[theorem]{Example}

    \newtheorem{remark}[theorem]{Remark}
    
        \newtheorem{question}[theorem]{Question}
\theoremstyle{remark}

\DeclareMathOperator{\id}{id}

\DeclareMathOperator{\Ell}{Ell}

\title[Constructions in minimal dynamics and applications to $\mathrm{C}^*$-algebras]{Constructions in minimal amenable dynamics and applications to the classification of $\mathrm{C}^*$-algebras}

\author[R.J. Deeley, I.F. Putnam, K.R. Strung]
{Robin J. Deeley \and
Ian F. Putnam \and
Karen R. Strung}
\address{Department of Mathematics,
University of Colorado Boulder
Campus Box 395,
Boulder, CO 80309-0395, USA }
\email{robin.deeley@gmail.com}

\address{Department of Mathematics and Statistics,
University of Victoria,
Victoria, B.C., Canada V8W 3R4} 
\email{ifputnam@uvic.ca}

\address{Institute for Mathematics, Astrophysics, and Particle Physics, Radboud University, Postbus 9010, 6500 GL Nijmegen,
The Netherlands}
\email{k.strung@math.ru.nl}

\date{\today}
\subjclass[2010]{37B05, 46L35, 46L85, 19K99}
\keywords{minimal homeomorphisms, equivalence relations, classification of nuclear \mbox{$\mathrm{C}^{*}$-algebras}}
\thanks{KRS is funded by
Sonata 9 NCN grant 2015/17/D/ST1/02529 and a Radboud Excellence Initiative Postdoctoral Fellowship. IFP is supported in part by an NSERC Discovery Grant.}

\begin{document}

\begin{abstract} We study the existence of minimal dynamical systems, their orbit and minimal orbit-breaking equivalence relations, and their applications to $\mathrm{C}^*$-algebras and $K$-theory.  We show that given any finite CW-complex there exists a space with the same $K$-theory and cohomology that admits a minimal homeomorphism. The proof relies on the existence of homeomorphisms on  point-like spaces constructed by the authors in previous work, together with existence results for skew product systems due to Glasner and Weiss. To any minimal dynamical system one can associate minimal equivalence relations by breaking orbits at small subsets. Using Renault's groupoid $\mathrm{C}^*$-algebra construction we can associate $K$-theory groups to minimal dynamical systems and orbit-breaking equivalence relations. We show that given arbitrary countable abelian groups $G_0$ and $G_1$ we can find a minimal orbit-breaking relation such that the $K$-theory of the associated $\mathrm{C}^*$-algebra is exactly this pair. These results have important applications to the Elliott classification program for $\mathrm{C}^*$-algebras. In particular, we make a step towards determining the range of the Elliott invariant of the $\mathrm{C}^*$-algebras associated to minimal dynamical systems with mean dimension zero and their minimal orbit-breaking relations.\end{abstract}

\maketitle

\setcounter{section}{-1}

\section{Introduction} \label{Sect:Intro}

Central to the study of topological dynamics are those systems which are minimal.  Given a compact metric space $X$ and a homeomorphism $\varphi : X \to X$, we say that $\varphi$ is minimal if $X$ contains no proper closed $\varphi$-invariant subsets, or, equivalently, that the $\varphi$-orbit of every point $x \in X$ is dense in $X$. It is well known that for any compact metric space $X$ and homeomorphism $\varphi : X \to X$, there is a minimal set $Y \subset X$ such that $\varphi$ restricts to a homeomorphism of $Y$ and that $\varphi|_Y$ is minimal.  Determining what spaces might arise as minimal sets, however, is quite difficult. An obvious question to ask is: Given an infinite compact metric space $X$, does $X$ admit a minimal homeomorphism? 

There is no known characterization of the spaces for which this might be possible, but in general, for sufficiently well-behaved spaces, the answer is likely to be ``no" and there are often explicit obstructions. 
For example given an absolute neighbourhood retract, the Lefschetz--Hopf theorem (see for example \cite{Brown:LefFixPt}) implies that the $K$-theory of the given space often contains enough information to conclude that any continuous self-map on it has a periodic  point. We are interested in infinite spaces, so this is enough to imply that such a space cannot admit a minimal homeomorphism. An explicit example of this type of result is \cite[Theorem 2]{Ful:PerPts}.  

Thus the answer of whether a given space admits a minimal homeomorphisms is usually negative. Instead we consider the following question.

\begin{question} \label{question1}
Let $W$ be a compact metric space. Does there exists a minimal dynamical system $(X, \varphi)$ where $X$ is an infinite compact metric space that agrees with $W$ on $K$-theory and \v{C}ech cohomology?
\end{question}

As a partial resolution, we show that this question has a positive answer when $W$ is a finite CW-complex. 

Closely related to minimal dynamical systems are minimal equivalence relations. Here, we equip a compact metric space $X$ with an equivalence relation and it is called minimal if every equivalence class is dense in $X$. Every minimal dynamical system $(X, \varphi)$ gives rise to a minimal equivalence relation where the equivalence classes are just the $\varphi$-orbits.  While in ergodic theory the analogous notions turn out to be essentially the same \cite{MR662736}, in topological dynamics, minimal equivalence relations are more general than the class of orbit equivalence relations associated to minimal homeomorphisms.

In this paper, we are interested in orbit-breaking relations. Here, the input is a minimal dynamical system $(X, \varphi)$ and nonempty closed subset $Y \subset X$. If $Y$ meets every $\varphi$-orbit at most once, then we can break the equivalence class corresponding to any orbit that passes through $Y$ into two distinct equivalence classes corresponding to the forward and backward orbits from $Y$. Since $Y$ meets every $\varphi$-orbit at most once, these equivalence classes are still dense in $X$, and, taken together with the equivalence classes corresponding to orbits that do not pass through $Y$, we obtain a new minimal equivalence relation on $X$. 

The question of what $K$-theory groups arise from the $\mathrm{C}^*$-algebras associated to minimal dynamical systems---from both orbit and orbit-breaking relations---is of particular importance to the Elliott classification program for $\mathrm{C}^*$-algebras. The Elliott classification program seeks to classify, up to $^*$-isomorphism, simple, separable, unital, and nuclear $\mathrm{C}^*$-algebras by the so-called Elliott invariant, an invariant consisting of $K$-theory and tracial data. The $\mathrm{C}^*$-algebras associated to minimal dynamical systems on infinite compact metric spaces are always simple, separable, unital, and nuclear, and as such, have long provided the classification program with interesting and elegant examples.

There are many such orbit-breaking relations for which the associated $\mathrm{C}^{*}$-algebras
cannot be isomorphic to the $\mathrm{C}^*$-algebra of any minimal dynamical system. Breaking at a single point in a Cantor minimal system, for example, results in an approximately finite (AF) algebra, which has trivial $K_1$-group. The $\mathrm{C}^*$-algebra of any minimal dynamical system, on the other hand, always has nontrivial $K_1$-group.

A recent spectacular achievement of the Elliott program is the classification of all simple, separable, unital, infinite-dimensional $\mathrm{C}^*$-algebras with finite nuclear dimension that satisfy the Universal Coefficient Theorem (UCT). The nuclear dimension is a $\mathrm{C}^*$-algebraic version of topological covering dimension and when a $\mathrm{C}^*$-algebra has finite nuclear dimension, its structure is much better behaved. The UCT is a tool which, loosely speaking, allows one to transfer information between $KK$-theory and $K$-theory. While it is a crucial assumption, it holds for all known examples of nuclear $\mathrm{C}^*$-algebras, in particular for those which arise from minimal dynamical systems by \cite{Tu:Groupoids}. Let us call a $\mathrm{C}^*$-algebra \emph{classifiable} if it is simple, separable, unital, infinite-dimensional, has finite nuclear dimension, and satisfies the UCT.  Based on this recent success, it is natural to ask the following:
 
\begin{question} \label{q2}
Which classifiable $\mathrm{C}^{*}$-algebras arise as $\mathrm{C}^*$-algebras of minimal \'{e}tale equivalence relations?
\end{question}

The $\mathrm{C}^*$-algebras associated to minimal dynamical systems---from both orbit and orbit-breaking relations---are both examples of minimal \'{e}tale equivalence relations \mbox{$\mathrm{C}^*$-algebras}, also called principal \'etale groupoids. As such, although Questions \ref{question1} and \ref{q2} are not equivalent, they are closely related. Similar work done in this direction is \cite{Li:Cartan} (also see \cite{AusMitra:GrpModGelfDual, BarLi:CartanUCT1, BarLi:CartanUCT2}). The main result of Li in \cite{Li:Cartan} is that every stably finite classifiable $\mathrm{C}^*$-algebra can be realized as the $\mathrm{C}^*$-algebra of a \emph{twisted} principal \'etale groupoid.  There, rather than focusing on constructions coming directly from dynamics, Li mimics known inductive limit constructions at the level of the groupoids. Another difference is that the twist on the grouping is only nontrivial when the $K_0$-group of the corresponding $\mathrm{C}^*$-algebra is torsion free. Below, although our constructions do not realize the Elliott invariant of every classifiable $\mathrm{C}^*$-algebra, we are able to have torsion in $K$-theory without requiring any twists. In Subsection~\ref{Subsect:RR0}, we improve on the results of \cite{MR3770169} for classifiable $\mathrm{C}^*$-algebras with real rank zero by allowing for torsion. 

Constructions of Giol and Kerr show that there are examples of crossed products associated to minimal homeomorphisms of infinite compact metric spaces which have infinite nuclear dimension \cite{GioKerr:Subshifts}. Thus not all such crossed product $\mathrm{C}^*$-algebras---or their orbit-breaking subalgebras---will fall within the scope of the classification theorem. However, Elliott and Niu provided a sufficient condition for finite nuclear dimension when they showed that whenever $(X, \varphi)$ has mean dimension zero, the crossed product will be isomorphic to itself tensored with the Jiang--Su algebra, $\mathcal{Z}$ \cite{EllNiu:MeanDimZero}. The Jiang--Su algebra is a simple, separable, unital, nuclear, infinite-dimensional $\mathrm{C}^*$-algebra whose Elliott invariant is the same as the Elliott invariant of $\mathbb{C}$.   A $\mathrm{C}^*$-algebra $A$ is called $\mathcal{Z}$-stable if it absorbs the Jiang--Su algebra $\mathcal{Z}$ tensorially, $A \cong A \otimes \mathcal{Z}$. Finite nuclear dimension of a \mbox{$\mathrm{C}^*$-algebra} $A$ turns out to be equivalent to $\mathcal{Z}$-stability of $A$, see \cite{CETWW} (which was based on earlier work in \cite{MR3273581, MR3418247}). When combined with the theorem of Elliott and Niu, results of Archey, Buck and Phillips imply that mean dimension zero ensures that any orbit-breaking algebra is also $\mathcal{Z}$-stable and hence has finite nuclear dimension \cite{ArBkPh-Z}.

It follows from this that a classification theorem for the $\mathrm{C}^*$-algebras associated to minimal dynamical systems is close at hand. One fundamental remaining piece, is determining the range of the Elliott invariant for such $\mathrm{C}^*$-algebras. In particular, it is directly related to answering Question~\ref{q2}, above.  A key step forward to understanding the range was the authors' construction of minimal homeomorphisms on ``point-like" spaces---spaces whose $K$-theory and cohomology are the same as a point \cite{DPS:DynZ}.  We can furthermore arrange that such a system is uniquely ergodic. After breaking the orbit at a single point, we are left with a minimal equivalence relation whose $\mathrm{C}^*$-algebra is $^*$-isomorphic to the Jiang--Su algebra. This system plays an important role in this paper. Using a skew-product construction of Glasner and Weiss we are able to use this system to produce many new minimal systems of mean dimension zero on spaces with the $K$-theory and \v{C}ech cohomology of any finite CW complex. 

The paper is structured as follows.  Section~\ref{Sect:Prelim} contains the background required for the paper.  For the most part this is $\mathrm{C}^*$-algebraic in nature and the reader only interested in the dynamical aspects of the paper can skip to Sections~\ref{Sect:MinSys} and \ref{Sect:MD}, which are the purely dynamical sections of the paper. These sections contain the construction of minimal dynamical systems and the mean dimension of these systems respectively. In particular, the details of the positive answer to Question \ref{question1} for finite CW-complexes can be found in Section 3.2. The second, mostly $\mathrm{C}^*$-algebraic part of the paper consists of Sections~\ref{Sect:K} through \ref{Sect:Class}. In Section~\ref{Sect:K}, we move to $\mathrm{C}^*$-algebras associated to minimal dynamical systems and record three lemmas about their $K$-theory which are fundamental for the results in the subsequent sections. The $\mathrm{C}^*$-algebras associated to the minimal skew product dynamical systems constructed in Subsection~\ref{minSkewProSec} are studied in Section~\ref{Sect:Skew}. Finally, Section~\ref{Sect:Class} contain applications to the Elliott classification program; the first subsection looks at $\mathrm{C}^*$-algebras with only the trivial projections while the second subsection considers $\mathrm{C}^*$-algebras with real rank zero.

\section{Preliminaries} \label{Sect:Prelim}

By a \emph{dynamical system}, we mean a compact  Hausdorff space $X$, which for the purposes of this paper is always assumed to be metrizable, equipped 
with a homeomorphism $\varphi : X \to X$.  The \emph{$\varphi$-orbit} of a point $x$ in $X$ is the set $\{ \varphi^{n}(x) 
\mid n \in \mathbb{Z} \}$.

In the sequel, $\varphi$ will always induces a \emph{free} action of the integers on $X$, which is to say that if there exists $x \in X$ with $\varphi^{n}(x)=x$, then $n=0$. The homeomorphism $\varphi$ is \emph{minimal}, or $(X, \varphi)$ is a \emph{minimal dynamical system}, if the only closed subsets $Y \subseteq X$ with $\varphi(Y)=Y$ are $X$ and the empty set. This is equivalent to the condition that every  $\varphi$-orbit is dense in $X$. If $X$ is infinite, then any minimal dynamical system induces a free action.
 
Given two dynamical systems $(X, \varphi)$ and $(Y, \psi)$ a map $\pi : X \to Y$ is called  a \emph{factor map} if $\pi$  is a continuous surjection satisfying $\pi \circ \varphi = \psi \circ \pi$.   In this case, $(Y, \psi)$ is called a \emph{factor} of $(X, \varphi)$ and   $(X, \varphi)$ is a called an \emph{extension} of $(Y, \psi)$. If $\pi : (X, \varphi) \to (Y, \psi)$ is a factor map then  $\pi : (X, \varphi^n) \to (Y, \psi^n)$ is also a factor map for any integer $n$. A factor map $\pi : X \to Y$ is 
\emph{almost one-to-one} if it is one-to-one on a residual subset of $X$.

Any minimal dynamical system $(X, \varphi)$ has an associated topological groupoid, the \emph{transformation groupoid} $X \rtimes_{\varphi} \mathbb{Z}$, see for example \cite{MR584266}. Here, since our systems will always be free,  it will be more convenient to reformulate the transformation groupoid as the orbit equivalence relation on $X$. Given a free dynamical system $(X, \varphi)$,  define the orbit equivalence relation 
\[
\mathcal{R}_{\varphi} := \{ (x, \varphi^{n}(x) ) \mid x \in X, n \in \mathbb{Z} \},
\]
which is an equivalence relation whose equivalence classes are simply the $\varphi$-orbits. As  the dynamical system is free, the map 
\[ X \rtimes_{\varphi} \mathbb{Z} \to \mathcal{R}_{\varphi}, \quad  (x, n) \mapsto (x, \varphi^{n}(x)) \] is 
 a bijection. We endow $X \rtimes_{\varphi} \mathbb{Z}$ with the product topology and equip $\mathcal{R}_{\varphi}$ with a topology via this map, 
 that is, $\mathcal{R}_{\varphi}$ is given the unique topology which makes this map a 
 homeomorphism. This endows the equivalence relation $\mathcal{R}_{\varphi}$ with an \'etale topology: the topology on an equivalence relation $\mathcal{R} \subset X \times X$ is \'etale if the maps $\mathcal{R} \to X$ given by $(x,y) \mapsto x$ and $(x,y) \mapsto y$ are local homeomorphisms. (Note that in the topological groupoid literature the term \emph{equivalence relation} is sometimes reserved for an equivalence relation $\mathcal{R} \subset X \times X$ with topology inherited from the product topology on $X \times X$, whereas an equivalence relation equipped with any another topology is called a \emph{principal groupoid}.)
 
An equivalence relation on a compact metric space $X$ is \emph{minimal} if, for every $x \in X$, the equivalence class of $x$ is dense in $X$.  Observe that if $(X, \varphi)$ is a free dynamical system, then $\mathcal{R}_{\varphi}$ is minimal if and only if $(X, \varphi)$ is minimal.

Suppose $(X, \varphi)$ is a minimal dynamical system and $Y \subseteq X$ is a closed 
non-empty subset of $X$. We say that $Y$ \emph{meets every orbit at most once} if $\varphi^n(Y) \cap Y = \emptyset$ for each $n \neq 0$. Given a minimal dynamical system $(X, \varphi)$ and closed non-empty subset $Y \subset X$ meeting every orbit at most once, we construct another equivalence relation using the groupoid construction of \cite[ Example 2.6]{Put:K-theoryGroupoids}, 
which was originally used in the $\mathrm{C}^*$-algebraic setting in \cite{Putnam:MinHomCantor}. 
 Define $\mathcal{R}_Y \subseteq \mathcal{R}_{\varphi}$ to be the subequivalence relation obtained from splitting every orbit that passes through $Y$ into two equivalence classes. Specifically,
 \[ 
 \mathcal{R}_{Y} = \mathcal{R}_{\varphi} \setminus \{ (\varphi^{k}(y), \varphi^{l}(y)) 
 \mid y \in Y, \quad  l < 1 \leq k \text{ or }  k < 1 \leq l \}.
 \]
 
 It is a simple matter to check that this is an open  subequivalence relation of $\mathcal{R}_{\varphi}$ and  hence is also \'{e}tale. If some  $\varphi$-orbit does not meet $Y$, then  it is an equivalence class in both $\mathcal{R}_{\varphi}$ and $\mathcal{R}_{Y}$. If the  orbit does meet $Y$, say at the point $y$, then its $\varphi$-orbit becomes two distinct equivalence classes in $\mathcal{R}_{Y}$, namely $\{ \varphi^{n}(y) \mid n \geq 1 \}$ and   $\{ \varphi^{n}(y) \mid n \leq 0 \}$. In this sense, the orbit is ``broken'' in two at the point $y$.

Since $X$ is compact, for any point $x \in X$ both its forward orbit and backward orbit are dense in $X$. Thus we make the following observation.

\begin{proposition} \label{equDense}
If $(X, \varphi)$ is a minimal system, $X$ is infinite and $Y$ is a closed non-empty subset of $X$ that meets each orbit at most once, then
$\mathcal{R}_Y$ is minimal. 
\end{proposition}

Let $\mathcal{R} \subset X \times X$ be an equivalence relation equipped with an \'etale topology, such as $\mathcal{R}_{\varphi}$ or $\mathcal{R}_{Y}$.  Using the method of Renault in  \cite[Chapter II]{MR584266} we  construct the reduced groupoid $\mathrm{C}^*$-algebra $\mathrm{C}^*_r(\mathcal{R})$ as follows. Equip the linear space $C_c(\mathcal{R})$ of compactly supported continuous functions $\mathcal{R} \to \mathbb{C}$ with a product and involution given by
\begin{eqnarray}
 (fg)(x, x') &=& \sum_{\substack{y \in X\\(x,y), (y,x') \in \mathcal{R}}} f(x, y)g(y,x'), \\
(f^*)(x,x') &=& \overline{f(x',x)},
\end{eqnarray}
for $f, g \in C_c(\mathcal{R})$ and $(x,x') \in \mathcal{R}$. This makes $C_c(\mathcal{R})$ into a $^*$-algebra. Let $\ell^2(\mathcal{R})$ denote the Hilbert space of square summable functions on $\mathcal{R}$. Then we define the regular representation $\lambda : C_c(\mathcal{R}) \to \mathcal{B}(\ell^2(\mathcal{R}))$ 
\[ (\lambda(f) \xi)(x, x') =  \sum_{\substack{y \in X\\(x,y), (y,x') \in \mathcal{R}}} f(x, y)\xi(y,x'),\]
for $f \in  C_c(\mathcal{R})$, $\xi \in \ell^2(\mathcal{R})$ and $(x,x') \in \mathcal{R}$. The reduced groupoid $\mathrm{C}^*$-algebra $\mathrm{C}^*_{r}(\mathcal{R})$ is then the closure of $\lambda(C_c(\mathcal{R}))$ with respect to the norm on $\mathcal{B}(\ell^2(\mathcal{R}))$. There is also a full groupoid $\mathrm{C}^*$-algebra, however if the  groupoid is amenable its  full and reduced $\mathrm{C}^{*}$-algebras coincide. In this case, we will simply denote the groupoids $\mathrm{C}^*$-algebra as $\mathrm{C}^*(\mathcal{R})$.  For any dynamical system $(X, \varphi)$, the equivalence relation $\mathcal{R}_{\varphi}$ is always amenable \cite[Example II.3.10]{MR584266}. When $Y$ is a closed nonempty subset meeting every $\varphi$-orbit at most once, then $\mathcal{R}_Y$ is also amenable, since it is an open subequivalence relation of $\mathcal{R}_{\varphi}$ \cite[Proposition 5.1.1]{MR1799683}.

Let $(X, \varphi)$ be a minimal dynamical system and a nonempty closed subset $Y$ meeting every $\varphi$-orbit at most once. The $\mathrm{C}^*$-algebra associated with $\mathcal{R}_{\varphi}$ is isomorphic to the crossed product $\mathrm{C}^*$-algebra $C(X) \rtimes_{\alpha} \mathbb{Z}$. Since $\mathcal{R}_{Y}$ is an open subequivalence relation of $\mathcal{R}_{\varphi}$, there is an inclusion  $C_c(\mathcal{R}_{Y} ) \subseteq  C_c(\mathcal{R}_{\varphi})$  by simply setting a function $f \in C_c(\mathcal{R}_Y)$ to zero on $\mathcal{R}_{\varphi} \setminus \mathcal{R}_Y$. This extends to a  unital inclusion of $\mathrm{C}^*$-algebras. If $A := \mathrm{C}^*(\mathcal{R}_{\varphi}) \cong C(X) \rtimes_{\varphi} \mathbb{Z}$, we will usually denote $\mathrm{C}^*(\mathcal{R}_Y)$ by $A_Y$.  Observe that both $A$ and $A_Y$ contain copies of $C(X)$, so, taken together with the inclusion $\iota : A_Y \into A$,  we have a commutative diagram
\begin{displaymath} 
\xymatrix{ & C(X) \ar[dl]_{i_1} \ar[dr]^{i_2} &\\
A_{\tilde{W}} \ar[rr]_{\iota} & & A.}
\end{displaymath}
 
Another description of the $\mathrm{C}^*$-subalgebra $A_{Y}$ often encountered in the literature  is given by the $\mathrm{C}^*$-subalgebra of $A$ generated by subsets $C(X)$ and $u C_0(X \setminus Y)$,
\[
A_Y = C^*(C(X), uC_0(X \setminus Y)) \subseteq C(X) \rtimes_{\varphi} \mathbb{Z}, 
\]
where $C(X) \subset A$ is the standard inclusion
and $u$ is the unitary inducing $\varphi$, that is, the unitary $u \in A$ satisfying $ ufu^* = f \circ \varphi^{-1}$ whenever $ f\in C(X)$.

An important invariant for dynamical systems is its mean dimension. Mean dimension, introduced by Gromov \cite{MR1742309}, is an invariant of dynamical systems akin to a dynamical version of covering dimension. We refer the reader to \cite{EllNiu:MeanDimZero, LinWeiss:MTD}
or to a more complete discussion in Section 4. In light of the next result, here we are interested in systems with mean dimension zero.

\begin{theorem} \label{classifiableTheorem}
Let $(X, \varphi)$ be a minimal dynamical system with $X$ an infinite compact metric space. Suppose that $(X, \varphi)$ has mean dimension zero. Then $A := C(X) \rtimes_{\varphi} \mathbb{Z}$ is classifiable. Furthermore, if $Y \subset X$ is a nonempty closed subset meeting every orbit at most once, then $A_Y$ is also classifiable.
\end{theorem}

\begin{proof}
 The first statement follows from the results of \cite{EllNiu:MeanDimZero}, where it is shown that mean dimension zero implies $\mathcal{Z}$-stability of the crossed product.  That $A_Y$ is separable and unital is clear. Since $\mathcal{R}_Y$ is amenable, $A_Y$ is nuclear (see for example \cite[Corollary 6.2.14]{MR1799683}) and satisfies the UCT \cite{Tu:Groupoids}. Moreover, since $Y$ meets every $\varphi$-orbit at most once, Proposition \ref{equDense} implies that $A_Y$ is a simple $\mathrm{C}^*$-algebra. Finally, since $A_Y$ is a centrally large subalgebra of $A$ (see \cite[Section 4]{ArcPhi2016}) it follows that $A_Y$ is also $\mathcal{Z}$-stable \cite[Theorem 2.4]{ArBkPh-Z}.
\end{proof}

We are particularly interested 
in the $K$-theory of these $\mathrm{C}^*$-algebras. The principal tool
for the computation of $K_{*}(C(X) \rtimes_{\varphi} \mathbb{Z})$ is the Pimsner--Voiculescu exact sequence, which is given by
\begin{displaymath} 
\xymatrix{
K^0(X) \ar[r]^-{\id-\varphi_*} & K^0(X) \ar[r] &  K_0(C(X) \rtimes_{\varphi} \Z)  \ar[d]^{\partial_{\rm PV}}\\
 K_1(C(X) \rtimes_{\varphi} \Z) \ar[u]^{\partial_{\rm PV}} &  K^1(X) \ar[l] & K^1(X) \ar[l]_-{\id-\varphi_*}.}
\end{displaymath}

For orbit-breaking subalgebras $\mathrm{C}^*(\mathcal{R}_Y)$, we use the following exact sequence which relates the $K$-theory of $\mathrm{C}^*(\mathcal{R}_Y)$ to the $K$-theory of  $C(Y)$ and $C(X) \times_{\varphi} \mathbb{Z}$.

\begin{theorem}[Theorem 2.4 and Example 2.6 of \cite{Put:K-theoryGroupoids}] \label{PutExtSeq} 
Let $(X, \varphi)$ be a minimal dynamical system and $Y \subseteq X$ a closed non-empty subset of $X$ which meets 
every orbit at most once. Let $ A:= C(X) \rtimes_{\varphi} \mathbb{Z}$ and let $\iota : A_Y \into A$ be the inclusion map. Then there exists a six-term exact sequence
\begin{displaymath} 
\xymatrix{ K^0(Y) \ar[r] & K_0(A_Y) \ar[r]^-{\iota_*} & K_0(C(X) \rtimes_{\varphi} \mathbb{Z}) \ar[d]^{\partial_{\rm OB}}\\
K_1(C(X) \rtimes_{\varphi} \mathbb{Z} ) \ar[u]^{\partial_{\rm OB}} & K_1(A_Y) \ar[l]_-{\iota_*} & K^1(Y). \ar[l]}
\end{displaymath}
\end{theorem}

The details for the remaining maps in this six-term exact sequence are given in \cite{Put:K-theoryGroupoids} and also discussed further in Section~\ref{Sect:K}.

In addition to the six-term exact sequence which allows us to relate the $K$-theory of an orbit-breaking subalgebra $A_Y$ to the containing crossed product $A = C(X) \rtimes_{\varphi} \mathbb{Z}$, we can also compare their tracial state spaces. In fact, by \cite[Theorem 6.2, Theorem 7.10]{Phi:LargeSubalgebras}, the restriction map 
\[ T(A) \to T(A_Y), \quad \tau \mapsto \tau|_{A_Y} \]
is bijective.

\section{Minimal systems} \label{Sect:MinSys}

In this section we discuss three methods for constructing minimal dynamical systems that have properties with interesting implications for the associated $\mathrm{C}^*$-algebras. First, we recall the construction of minimal homeomorphisms on point-like spaces from \cite{DPS:DynZ}. Then, by constructing minimal skew product systems we consider Question \ref{question1} in the case of a finite CW-complex.  \vspace{0.1cm}\\
{\bf Question:} Given $W$ a finite CW-complex does there exist a compact metric space $\widetilde{W}$ with  $K^*(\widetilde{W}) \cong K^*(W)$ that admits a minimal homeomorphism? \vspace{0.1cm} \\
Theorem~\ref{Thm:ExistenceOfMinHomeo} establishes a positive answer. We note that since we are only interested in the $K$-theory of $W$, we can and will assume $W$ is connected. Finally, we recall a construction of minimal homeomorphisms on nonhomogeneous metric spaces from \cite{DPS:nonHom} which will allows us to show in Section~\ref{Subsect:RR0} that there exist orbit-breaking subalgebras with interesting $K$-theory and many projections.

\subsection{Minimal dynamical systems on point-like spaces} \label{constructionZ}
In \cite{DPS:DynZ}, the authors constructed a minimal homeomorphisms on  infinite ``point-like'' spaces; that is,  infinite compact connected metric spaces that have both the same \v{C}ech cohomology and topological $K$-theory as a point and moreover have finite covering dimension \cite[Corollary 1.12]{DPS:DynZ}. As these systems will play the main role in the sequel, we review some of their properties.

A generalized cohomology theory is called \emph{continuous} if, given an inverse limit of spaces, there exists an associated inverse limit at the level of the cohomology theory. The interest reader can find more on this notion in \cite[Section 21.3]{Bla:K-theory} (note that in \cite{Bla:K-theory} these notions are formulated in $\mathrm{C}^*$-algebraic terms). Two examples of continuous generalized cohomology theories are \v{C}ech cohomology and $K$-theory. These are the most relevant in this paper.

The existence of minimal diffeomorphisms of odd dimensional spheres of dimension at least three was proved by Fathi and Herman in the uniquely ergodic case \cite{FatHer:Diffeo} and later generalized by Windsor \cite{Wind:not_uniquely_ergo} to minimal diffeomorphisms with a prescribed number of ergodic measures.  In \cite{DPS:DynZ}, the authors showed that such a minimal diffeomorphism can be used to construct minimal dynamical systems on a point-like space $Z$. Given a minimal diffeomorphism $\varphi : S^d \to S^d$, $d \geq 3$ odd, the associated space $Z$ is constructed by removing a subset $ L_{\infty}$ that is a $\varphi$-invariant immersion of $\mathbb{R}$ in $S^d$,  and completing $S^d \setminus L_{\infty}$ with respect to a metric obtained from the inverse limit structure. A homeomorphism $\zeta : Z \to Z$ is then given by extending the map $\varphi : S^d \setminus L_{\infty} \to S^d \setminus L_{\infty}$ to $Z$.  

There is a factor map $q : Z \to S^d$ which is one-to-one on $S^d \setminus L_{\infty}$, and every $\zeta$-invariant Borel probability measure $\mu$ satisfies $\mu(S^d \setminus L_{\infty}) = 1$. Further details for the factor map can be found in \cite[Corollary 1.16, Lemma 1.14]{DPS:DynZ}, and the reader is directed to the rest of the paper for the general construction of these minimal dynamical systems. We summarize the main aspects in the theorem below.

\begin{theorem} \label{ThmAboutZ}
Let $S^d$  be a sphere of odd dimension $d\geq 3$, and let $\varphi : S^d \to S^d$  be a minimal diffeomorphism. Then there exist an infinite compact metric space $Z$ with finite covering dimension and a minimal homeomorphism $\zeta : Z \to Z$ satisfying the following.
\begin{enumerate}
\item $Z$ is compact, connected, and homeomorphic to an inverse limit of contractible metric spaces $(Z_n, d_n)_{n \in \mathbb{N}}$.
\item For any continuous generalized cohomology theory we have an isomorphism $H^*(Z) \cong H^*(\{\mathrm{pt}\})$. In particular this holds for \v{C}ech cohomology and $K$-theory. \label{sameCohomAsPoint}
\item There is an almost one-to-one factor map $q : Z \to S^d$ which induces a bijection between $\zeta$-invariant Borel probability measures on $Z$ and $\varphi$-invariant Borel probability measures on $S^d$. \label{FactorMap}
\end{enumerate} 
\end{theorem}

\subsection{Minimal skew products} \label{minSkewProSec}

In \cite{GlaWei:MinSkePro}, Glasner and Weiss give conditions for when one may obtain a minimal dynamical system from skew products. We are interested in skew product systems arising from a minimal homeomorphism on a point-like space. First, let us recall some notation from \cite{GlaWei:MinSkePro}. For a compact metric space $X$ with metric $d$, let ${\rm Homeo}(X)$ denote the space of homeomorphisms of $X$ equipped with the metric, which in an abuse of notation we also denote by $d$, given by
\[ d(g,h) = \sup_{x \in X} d(g(x), h(x)) + \sup_{x \in X} d(g^{-1}(x), h^{-1}(x)).\]

 Let $(Z, \zeta)$ be a minimal dynamical system given by Theorem~\ref{ThmAboutZ}. For a compact metric space $Y$, let $X:=Z\times Y$. Define a subset of ${\rm Homeo}(X)$ by
\[
\mathcal{O}(\zeta\times \id) =\{ G^{-1} \circ (\zeta \times \id )\circ G \: | \: G \in {\rm Homeo}(X). \}
\]
We are also be interested in subsets of $\mathcal{O}(\zeta\times id)$ of the form
\[
\mathcal{S}(\zeta\times \id) =\{ G^{-1} \circ (\zeta \times \id )\circ G \: | \: G \in {\rm Homeo}(X) \hbox{ such that }G(z,y)=(z, g_z(y)) \},
\]
where $z \mapsto g_z$ is a continuous map from $Z$ to ${\rm Homeo}(Y)$.

Applying Theorem 1 of \cite{GlaWei:MinSkePro} gives the following result:

\begin{theorem} \label{GlasnerWeissToZ}
Suppose $Y$ is a compact metric space and that $\mathcal{G}$ is a path connected subgroup of ${\rm Homeo}(Y)$ such that $(Y, \mathcal{G})$ is minimal. Then there is a residual subset of $\overline{\mathcal{O}(\zeta\times \id)}$ which consists entirely of minimal homeomorphisms. Likewise, there is a residual subset of $\overline{\mathcal{S}(\zeta\times \id)}$ which consists entirely of minimal homeomorphisms.
\end{theorem}

The reader can find a list of spaces that satisfy the hypotheses of the previous theorem on page 7 of \cite{DirMal:NonInvMinSkePro}. Although many spaces are covered, it cannot be applied directly to any finite CW-complex. However, if $W$ is a finite connected CW-complex the product of $W$ with the Hilbert cube $Q$ is a connected compact Hilbert cube manifold (see for example \cite[page 498]{West:HilbertCubeMflds}) and the hypotheses of Theorem \ref{GlasnerWeissToZ} are now satisfied. Since both the Hilbert cube $Q$ and $Z$ have the same $K$-theory and cohomology as a point, we arrive at the following.

\begin{theorem} \label{Thm:ExistenceOfMinHomeo}
Let $W$ be a finite connected CW-complex, and let $Q$ denote the Hilbert cube. Then $Z\times W\times Q$ admits a minimal homeomorphism and there are  isomorphisms
\[ H^*(Z\times W\times Q) \cong H^*(W), \quad K^*(Z \times W \times Q) \cong K^*(W)\]
of \v{Cech} cohomology and $K$-theory.
\end{theorem}

\begin{proof}
The existence of the minimal homeomorphism follows using Theorem \ref{GlasnerWeissToZ}. The second statement follows from the K\"unneth formula, the fact that $Q$ is contractible, and Theorem \ref{ThmAboutZ} (\ref{sameCohomAsPoint}).
\end{proof}

\begin{example} \label{uniqueErgodic}
In the case that the system $(Z, \zeta)$ is uniquely ergodic, then  Theorem 2 in \cite{GlaWei:MinSkePro} tells us that  we can obtain a uniquely ergodic skew product systems. In this situation, the crossed product $C(Z \times W \times Q)\rtimes_{\tilde{\zeta}} \mathbb{Z}$ will have a unique tracial state and the associated crossed products are classifiable because uniquely ergodic systems always have mean dimension zero. 
\end{example}

\begin{example}
If $M$ is a closed connected manifold with finite dimension, then we can apply Theorem \ref{GlasnerWeissToZ} directly to $X=Z\times M$. In this case, $X$ is finite dimensional. Hence the resulting minimal system has mean dimension zero from which it follows that the crossed product algebra fits within classification. In the general (i.e., finite CW-complex) case, we must include the infinite dimensional Hilbert cube and hence classification becomes unclear. This will be discussed in more detail in Section \ref{Sect:MD}.
\end{example}

\begin{proposition} \label{alphaActZeroKtheory}
Suppose $\alpha: Z\times W\times Q \rightarrow Z\times W\times Q$ is a minimal homeomorphism constructed as in Theorem \ref{GlasnerWeissToZ}. Then $\alpha$ induces the identity map on $K$-theory.
\end{proposition}

\begin{proof}
Since $\alpha$ is in the closure of $\mathcal{O}(\zeta\times \id) =\{ G^{-1} \circ (\zeta \times \id )\circ G \: | \: G \in {\rm Homeo}(X) \}$, we need only show that elements of $\mathcal{O}(\zeta \times \id)$ act as the identity on $K$-theory. This follows since $\zeta^*: K^*(Z) \rightarrow K^*(Z)$ is the identity map, as is shown in the proof of  \cite[Proposition 2.8]{DPS:DynZ}). Thus $(G^{-1} \circ (\zeta \times \id )\circ G)^*=G^* \circ id \circ (G^*)^{-1}=\id$.
\end{proof}

A similar but slightly different construction is also possible which will allow us to say more about invariant measures of the minimal dynamical system.   Let $(S^d, \varphi)$, $d \geq 3$ odd,  be a minimal diffeomorphism and $(Z, \zeta)$ the corresponding point-like system given by Theorem~\ref{ThmAboutZ}. For a finite CW-complex $W$ and the Hilbert cube $Q$, we apply \cite[Theorem 1]{GlaWei:MinSkePro} to $S^d\times W \times Q$ to obtain a minimal homeomorphism 
\[
\tilde{\varphi} : S^d\times W \times Q \rightarrow S^d\times W \times Q, \qquad
(s, w) \mapsto (\varphi(s), h_s (w)),
\]
where $s \in S^d$, $w\in W\times Q$ and $h : S^d \rightarrow {\rm Homeo}(W\times Q)$. Let $q : Z \to S^d$ be the factor map of Theorem~\ref{ThmAboutZ}~\ref{FactorMap}. Then we define a homeomorphism 
\[ \tilde{\zeta}: Z\times W\times Q \rightarrow Z\times W\times Q, \qquad (z, w) \mapsto (\zeta(z), h_{q(z)}(w)).
\]

\begin{proposition} \label{FacMapConTwo}
There is a factor map 
\[ \tilde{q} : (Z\times W\times Q, \tilde{\zeta}) \rightarrow (S^d\times W \times Q, \tilde{\varphi}) \]
defined by $\tilde{q}= q \times \id_{W\times Q}$.
\end{proposition}
\begin{proof} Since $q$ is a factor map, it is clear that $\tilde{q}= q \times \id_{W\times Q}$ is surjective. Also,
\begin{eqnarray*} 
\tilde{q} \circ \tilde{\zeta} (z, w) &=& \tilde{q}(\zeta(z), h_{q(z)}(w)) = (q (\zeta(z)), h_{q(z)}(w))\\
&=& (\varphi(q(z)), h_{q(z)}(w)) = \tilde{\varphi}(q(z), h_{q(z)}(w)),
\end{eqnarray*}
so $\tilde{q}$ intertwines the actions. Thus $\tilde{q}$ is a factor map.
\end{proof}

\begin{proposition} \label{MinMapConTwo}
The homeomorphism \[ \tilde{\zeta}: Z\times W\times Q \rightarrow Z\times W\times Q, \qquad (z, w) \mapsto (\zeta(z), h_{q(z)}(w))
\] is minimal. 
\end{proposition}
\begin{proof}
Let $L_{\infty} \subset S^d$ denote the $\varphi$-invariant immersion of $\mathbb{R}$ which is removed in the construction of $Z$. Suppose $F$ is a closed nonempty $\tilde{\zeta}$-invariant subset of $Z\times W\times Q$. Then $\tilde{q}(F)$ is a closed nonempty $\tilde{\varphi}$-invariant subset of $S^d\times W \times Q$ and since $\tilde{\varphi}$ is minimal, we have that $\tilde{q}(F)=S^d\times W \times Q$. By \cite[Lemma 1.14]{DPS:DynZ}, $q$ is injective when restricted to the $(S^d \setminus L_{\infty}) \subset Z$, from which it immediately follows that $\tilde{q}$ is injective on $(S^d \setminus L_{\infty}) \times W \times Q$. Hence ($S^d \setminus L_{\infty}) \times W \times Q \subseteq F$. However $(S^d \setminus L_{\infty}) \times W \times Q$ is dense in $Z\times W\times Q$. Hence $F$ is both closed and dense, so we conclude $F=Z\times W\times Q$, from which it follows that $\tilde{\zeta}$ is minimal.
\end{proof}

\begin{proposition}
The factor map $ \tilde{q} : (Z\times W\times Q, \tilde{\zeta}) \rightarrow (S^d\times W \times Q, \tilde{\varphi})$ induces a one-to-one bijection between the $\tilde{\zeta}$-invariant Borel probability measure on $Z\times W\times Q$ and the $\tilde{\varphi}$-invariant Borel probability measures on $S^d\times W\times Q$. 
\end{proposition}
\begin{proof}
Let $\mu$ be a $\tilde{\zeta}$-invariant measure. Then ${\tilde{q}}^*(\mu) =\mu\circ \tilde{q}^{-1}$ is $\tilde{\varphi}$-invariant. The set $L_{\infty} \times Q\times W$ is a Borel subset of $S^d \times W \times Q$, and since $L^{\infty}$ is $\varphi$-invariant, $L_{\infty} \times Q\times W$ is invariant under $\tilde{\varphi}$. As in the proof of \cite[Theorem 1.18]{DPS:DynZ}, it follows that $\tilde{q}^*(\mu)(L_{\infty} \times Q\times W)=0$. Hence $\mu((S^d \setminus L_{\infty}) \times W \times Q) = 1$. Since the factor map $q: Z \to S^d$ is one-to-one on $S^d \setminus L_{\infty}$, we have that $\tilde{q}$ is bijective on $(S^d \setminus L_{\infty}) \times W \times Q$, and the result follows.
\end{proof}

\subsection{Minimal homeomorphisms of nonhomogeneous spaces} \label{NonHomSec}

Many of the known examples of spaces admitting minimal homeomorphisms---such as the Cantor set, the odd-dimensional spheres, or the spaces of the previous section---are homogeneous spaces. Recall that a topological space $X$ is \emph{homogeneous} if, for every pair of points $x, y \in X$ there exists a homeomorphism $h : X \to X$ such that $h(x) = y$.  By contrast, the closed unit interval is nonhomogeneous: there cannot exist a homeomorphism of $[0,1]$ which takes $0$ to any point in the interior $(0,1)$, for example. In particular, any homeomorphism will be either an increasing map leaving the endpoints fixed, or a decreasing map interchanging $0$ and $1$. In any case, it is clear that there can exist no minimal homeomorphism of the unit interval. 

Indeed, it would appear at first that the existence of a minimal homeomorphism on a space implies some type of weak homogeneity. Nevertheless, Floyd constructed  a minimal system, 
$(\tilde{K}, \tilde{\varphi})$ with a factor map $\pi: (\tilde{K}, \tilde{\varphi}) 
\rightarrow (K, \varphi)$, where $(K, \varphi)$ is the $3^{\infty}$-odometer, a minimal dynamical system of the Cantor set $K$ \cite{floyd1949} by constructing a map on $K \times [0,1]$ whose minimal set is $\tilde{K}$.  For some 
$x \in K$, $\pi^{-1}\{ x \}$ is a single point, while for other 
$x \in K$, $\pi^{-1}\{ x \} \cong [0,1]$. Thus $\tilde{K}$ has some maximal connected
sets which are single points and hence zero dimensional, while other maximal connected sets have positive dimension. In particular, $\tilde{K}$ is nonhomogeneous.

This result was subsequently generalized by Gjerde and Johansen \cite{FloGjeJohSys}, who showed that given any Cantor minimal system $(K, \varphi)$, one can find a minimal dynamical system $(\tilde{K}, \tilde{\varphi})$ on a nonhomogeneous space $\tilde{K}$ as well as a factor map $\pi$ from $(\tilde{K}, \tilde{\varphi})$ onto $(K, \varphi)$, again by constructing $\tilde{K}$ as the minimal set of a homeomorphism on $K \otimes [0,1]$. They were able to allow for Cantor systems more general than the $3$-odometer by employing the use of Bratteli--Vershik models.

In \cite{DPS:nonHom}, the authors further generalize the construction of Floyd by not only replacing the Cantor $3$-odometer with more general minimal Cantor systems, but also by replacing the interval $[0,1]$ could be replaced by more complicated spaces, including spaces with arbitrarily large dimension. In particular, we have the following, which in Subsection~\ref{Subsect:RR0} will allow us to find minimal equivalence relations by breaking an orbit at any arbitrary finite dimensional, compact, connected metric space $Y$. We refer the reader to \cite{DPS:nonHom} for the details on this construction and in particular the proof of the next theorem.

\begin{theorem} \label{FloGjeJohSysThm}
Let $(K, \varphi)$ be a minimal system with $K$ the Cantor set and let $n \geq 1$ be a natural number. Then there exists 
 minimal system, 
$(\tilde{K}, \tilde{\varphi})$ with a factor map $\pi: (\tilde{K}, \tilde{\varphi}) 
\rightarrow (K, \varphi)$ such that, for some points $x$ in $K$, $\pi^{-1}\{ x \} \cong [0,1]^{n}$.
Moreover, the map $\pi$ induces an isomorphism $\pi^{*}:K^{*}(K) \rightarrow K^{*}(\tilde{K})$.
\end{theorem}

\section{Mean dimension} \label{Sect:MD}

In this section, we consider the issue of mean dimension for our minimal systems of the last section. As discussed in Section~\ref{Sect:Prelim}, we are  interested in this result because if a minimal dynamical system $(X, \varphi)$ has mean dimension zero, then the associated crossed product 
 $C(X) \rtimes_{\varphi} \mathbb{Z}$, as well as any orbit-breaking subalgebra,  is classifiable.

The point-like systems $(Z, \zeta)$ all have mean dimension zero because the spaces under consideration, $Z$, all have finite covering dimension. More specifically, if the system $(Z, \zeta)$ is an extension of $ \varphi: S^{d} \rightarrow S^{d}$,  then $Z$ has covering dimension either $d$ or $d-1$ (although we do not know which).  Similarly, suppose $(\tilde{K}, \tilde{\varphi})$ is a system constructed from a Cantor minimal system $(K, \varphi)$ and natural number $n$ using Theorem \ref{FloGjeJohSysThm}. Then $\dim(\tilde{K}) = n$, so $(\tilde{K}, \tilde{\varphi})$ will also have mean dimension zero. 

This leaves us to consider the  minimal skew products 
constructed via Theorem \ref{GlasnerWeissToZ}.
We show that generically they have mean dimension zero. The reader should note that 
if the original system is uniquely ergodic then generically the skew product system
 is also uniquely ergodic, see \cite[Theorem 2]{GlaWei:MinSkePro} and Example \ref{uniqueErgodic}.

We begin by fixing notation and developing some background material. Let $\mathcal{U}$ be a finite cover of an infinite compact metric space $X$. Then the \emph{order} of the cover $\mathcal{U}$, denoted ${\rm ord}(\mathcal{U})$, is given by
\[
{\rm ord}(\mathcal{U}) :=  \left( \max_{x\in X} \sum_{U \in \mathcal{U}} \chi_U(x) \right) -1,
\]
where $\chi_U$ is the characteristic function of the open set $U$. Another finite cover $\mathcal{V}$ is called a \emph{refinement} of $\mathcal{U}$ if each element in $\mathcal{V}$ is a subset of an element of $\mathcal{U}$. In that case, we write $\mathcal{V} \succ \mathcal{U}$. If $\mathcal{U}$ is a finite open cover, we define
\[
\mathcal{D}(\mathcal{U}) :=\min_{\mathcal{V} \succ \mathcal{U}} {\rm ord}(\mathcal{V}),
\]
where the minimum is taken over finite open refinements of $\mathcal{U}$.

Now suppose $\alpha \in {\rm Homeo}(X)$ and $n\in \N$. Then we define
\[
\mathcal{D}(\mathcal{U}, \alpha, n):=\mathcal{D}(\mathcal{U}\vee \alpha^{-1}(\mathcal{U}) \vee \cdots \vee \alpha^{-n+1}(\mathcal{U})),
\]
where $\vee$ denotes the join of the covers $\mathcal{U}$, $\alpha^{-1}(\mathcal{U}), \dots, \alpha^{-n+1}(\mathcal{U})$. Since $\mathcal{D}(\cdot)$ is subadditive with respect to joins of covers, 
\[ \lim_{n\rightarrow \infty} \frac{\mathcal{D}(\mathcal{U},\alpha, n)}{n} = \inf_{n\ge 1}  \frac{\mathcal{D}(\mathcal{U}, \alpha, n)}{n}\]
 exists. Define
\begin{equation}
\mathcal{D}(\mathcal{U}, \alpha):= \lim_{n\rightarrow \infty} \frac{\mathcal{D}(\mathcal{U}, \alpha, n)}{n}.
\end{equation}
The mean dimension of $(X, \alpha)$, denoted by ${\rm mdim}(\alpha)$, is then defined by 
\[
{\rm mdim}(\alpha) :=\sup_{\mathcal{U}} \mathcal{D}(\mathcal{U}, \alpha) \subset [0, \infty],
\]
where the supremum is taken over all finite open covers of $X$.
\begin{example}
Let $Y$ be a finite dimensional compact metric space and $\sigma : Y^{\Z} \rightarrow Y^{\Z}$ be the shift map. Then 
\begin{enumerate}
\item ${\rm mdim}(\sigma)={\rm dim}(Y)$;
\item $\sigma$ is the limit of $(\sigma_n)_{n\in \N}$ where for each $n$, ${\rm mdim}(\sigma_n)=0$.
\end{enumerate}
The first property is exactly \cite[Proposition 3.1]{LinWeiss:MTD}. For the second, let $\sigma_n : Y^{\Z} \rightarrow Y^{\Z}$ be defined via
\[
(x_l)_{l\in \N} \mapsto \left\{ \begin{array}{cc} x_l & |l|\ge n \\ x_{l-1} & -n < l < n+1 \\ x_n & l=-n \end{array} \right.
\]
For example, when $n=2$, we have
\[
\sigma_2((x_l)_{l\in \N}) = (\ldots x_{-4}x_{-3}x_2x_{-2}x_{-1}x_0x_{1}x_3x_4\ldots),
\]
where the $x_{-1}$ is in the $0$th position. One checks that
\[d(\sigma, \sigma_n) \rightarrow 0, \qquad n\rightarrow \infty,\]
and also that for each $n$, there exists $N$ such that $\sigma^N_n= \id$. Thus, using \cite[Proposition 2.7]{LinWeiss:MTD}  one has 
\[N{\rm mdim}(\sigma_n)={\rm mdim}(\sigma^N_n)={\rm mdim}(\id)=0.
\]
\end{example}

Now let us return to the systems constructed in Subsection~\ref{minSkewProSec}.  Let $(S^d, \varphi)$ be a diffeomorphism of an odd dimensional sphere, $d \geq 3$, and let $(Z, \zeta)$ be the associated point-like system. Let $Y$ be a compact metric space and set $X := Z \times Y$.

\begin{lemma} \label{mdimLemmaForS}
For any homeomorphism $\alpha \in \mathcal{O}(\zeta\times id) \subset{\rm Homeo}(X)$, we have ${\rm mdim}(\alpha)=0$. 
\end{lemma} 
\begin{proof}
By definition, $\alpha=G^{-1} \circ (\zeta \times \id) \circ G$ for some $G\in {\rm Homeo}(X)$. Hence $\alpha$ and $\zeta \times \id$ are conjugate, and so it is enough to show that $\zeta \times \id$ has mean dimension zero. Since $Z$ is finite dimensional, $\zeta:Z\rightarrow Z$ has mean dimension zero. The identity map always has mean dimension zero, so 
\[
{\rm mdim}(\zeta\times \id)\leq {\rm mdim}(\zeta) + {\rm mdim}(\id) =0,
\]
hence ${\rm mdim}(\alpha)=0$.
\end{proof}

The proof of the next lemma is standard and is therefore omitted. 
\begin{lemma} \label{closeUpToNthPower}
Suppose $\alpha \in {\rm Homeo}(X)$. Then for each $\epsilon>0$ and $N\in \N$, there exists $\delta>0$ such that if $ d(\alpha,\beta)< \delta$, then $d(\alpha^i, \beta^i)< \epsilon$ for each $i=1, \ldots, N$.
\end{lemma}

Fix a countable neighbourhood base for the topology on $X$ and let $\mathcal{U}$ be a finite open cover of $X$ whose elements are basic sets. For each such $\mathcal{U}$ and $k\in \N$, define 
\[
H_{\mathcal{U}, k} := \left\{ \alpha \in \overline{\mathcal{S}(\zeta \times \id)} \: \middle | \: \lim_{n\rightarrow \infty} \frac{\mathcal{D}(\mathcal{U}, \alpha, n)}{n} < \frac{1}{k} \right\}.
\]

\begin{lemma}
 Let $\mathcal{U}$ be a finite open cover of $X$ whose elements are basic sets and let $k \in \mathbb{N}$.  It holds that
\begin{enumerate}
\item the set $H_{\mathcal{U}, k}$ is dense in $\overline{\mathcal{S}(\zeta \times \id)}$;
\item the set $H_{\mathcal{U}, k}$ is open in $\overline{\mathcal{S}(\zeta \times \id)}$;
\item if $\alpha \in \bigcap_{\mathcal{U}, k} H_{\mathcal{U}, k}$, then ${\rm mdim}(\alpha)=0$.
\end{enumerate}
\end{lemma}

\begin{proof}
Fix $\mathcal{U}$ and $k$. If $\alpha \in \mathcal{S}(\zeta \times id)$ (rather than its closure) then $\alpha$ has mean dimension zero by Lemma \ref{mdimLemmaForS} and hence, for each $\mathcal{U}$, 
\[ \lim_{n\rightarrow \infty} \frac{\mathcal{D}(\mathcal{U}, \alpha, n)}{n} =0< \frac{1}{k}. \]
Thus $\mathcal{S}(\zeta \times \id) \subseteq H_{\mathcal{U}, k}$ and the first statement follows.

For the second statement, again fix $\mathcal{U}=\{ U_1, \ldots, U_m\}$ and $k$ and now also fix a homeomorphism $\alpha \in H_{\mathcal{U}, k}$. Then by definition, 
\[ \lim_{n\rightarrow \infty} \frac{\mathcal{D}(\mathcal{U}, \alpha, n)}{n}=L < \frac{1}{k} .\]
Therefore, there exists $N\in \N$ such that
\[
\frac{\mathcal{D}(\mathcal{U}, \alpha, N)}{N} < \frac{1}{k}.
\]
By \cite[Proposition 1.6.5]{Coornaert}, there exists a finite closed cover of $X$, call it $\mathcal{F}$, that refines $\mathcal{U} \vee \alpha^{-1}(\mathcal{U}) \vee \ldots \vee \alpha^{-N+1}(\mathcal{U})$ and $\mathcal{D}(\mathcal{U}, \alpha, N)={\rm ord}(\mathcal{F})$. 

Fix a closed set $F \in \mathcal{F}$. Then 
\begin{equation} \label{closedSetFasSubset}
F \subseteq U_{i_1} \cap \alpha^{-1}(U_{i_2}) \cap \ldots \cap \alpha^{-N+1}(U_{i_N}),
\end{equation}
where each $U_{i_r}$ is an open set in $\mathcal{U}$. For each $1\leq r \leq N$ we have $F \subseteq \alpha^{-r}(U_{i_r})$. Since $\alpha$ is a homeomorphism, $\alpha^{r}(F) \subseteq U_{i_r}$ and hence the closed sets $\alpha^r(F)$ and $X \setminus U_{i_r}$ are disjoint. The metric space $X$ is compact so both $\alpha^r(F)$ and $X \setminus U_{i_r}$ are compact. Therefore the distance between $\alpha^r(F)$ and $X \setminus U_{i_r}$ is equal to some $\epsilon_{F, r}>0$.

The collection $\mathcal{F}$ is finite so 
\[
\epsilon := {\rm min}\left\{ \frac{\epsilon_{F, r}}{2} \: \middle | \: F \in \mathcal{F} \hbox{ and } 1\le r \le N \right\}
\] 
exists and is a positive real number. By Lemma \ref{closeUpToNthPower}, there exists $\delta>0$ such that if $d(\alpha, \beta)< \delta$, then $d(\alpha^i, \beta^i)< \epsilon$ for each $i=1, \ldots, N$.

Now suppose $\beta \in \overline{\mathcal{S}(\zeta \times id)}$ with $d(\alpha, \beta)<\delta$. We show $\beta \in H_{\mathcal{U}, k}$. The first step is showing $\mathcal{F}$ refines $\mathcal{U} \vee \beta^{-1}(\mathcal{U}) \vee \ldots \vee \beta^{-N+1}(\mathcal{U})$. Fix $F\in \mathcal{F}$. By \eqref{closedSetFasSubset} we have $F \subseteq \alpha^{-r}(U_{i_r})$. Our goal is to show that $F\subseteq \beta^{-r}(U_{i_r})$ for each $r=1, \ldots, m$. Let $x \in F$. The definition of $\epsilon$ implies that $d(\alpha^r(y), X \setminus U_{i_r})\geq 2\epsilon$. Using this along with the fact that $d(\alpha^r, \beta^r)< \epsilon$, we have
\begin{align*}
2\epsilon & \leq d(\alpha^r(y), X\setminus U_{i_r}) \\
& \leq d(\alpha^r(y), \beta^r(y)) + d(\beta^r(y), X \setminus U_{i_r}) \\
& \leq \epsilon +  d(\beta^r(y), X \setminus U_{i_r})
\end{align*}
Thus $d(\beta^r(y), X \setminus U_{i_r})\geq \epsilon>0$. Hence $\beta^r(y)\notin X \setminus U_{i_r}$, from which it follows that $y \in \beta^{-r}(U_{i_r})$. Since $y$ and $r$ were arbitrary, $F \subseteq \beta^{-r}(U_{i_r})$ for each $r=1, \ldots, N$. Thus $\mathcal{F}$ refines $\mathcal{U} \vee \beta^{-1}(\mathcal{U}) \vee \ldots \vee \beta^{-N+1}(\mathcal{U})$. Using \cite[Proposition 1.6.5]{Coornaert}, one obtains
\[
\mathcal{D}(\mathcal{U}, \beta, N) \leq {\rm ord}(\mathcal{F}) = \mathcal{D}(\mathcal{U}, \alpha, N),
\]
and therefore 
\[
\frac{\mathcal{D}(\mathcal{U}, \beta, N)}{N} \leq \frac{\mathcal{D}(\mathcal{U}, \alpha, N)}{N} < \frac{1}{k}.
\]
Finally, 
\[\lim_{n \rightarrow \infty} \frac{\mathcal{D}(\mathcal{U}, \beta, n)}{n}=\inf_{n\geq 1} \frac{\mathcal{D}(\mathcal{U}, \beta, n)}{n} \leq \frac{\mathcal{D}(\mathcal{U}, \beta, N)}{N}< \frac{1}{k}.
\]
Thus $\beta \in H_{\mathcal{U}, k}$ which implies that $H_{\mathcal{U}, k}$ is open.

We now prove the final item in the statement. That is, if $\alpha \in \bigcap_{\mathcal{U}, k} H_{\mathcal{U}, k}$, then ${\rm mdim}(\alpha)=0$. Let $\mathcal{U}$ be any finite open cover of $X$. Then, using the definition of neighbourhood base, there exists $\mathcal{V}$ a finite open cover by basic sets that refines $\mathcal{U}$. Since $\mathcal{V} \succ \mathcal{U}$, for each $n$, 
\[ \mathcal{V} \vee \alpha^{-1}(\mathcal{V}) \vee \ldots \vee \alpha^{-n+1}(\mathcal{V}) \succ \mathcal{U} \vee \alpha^{-1}(\mathcal{U}) \vee \ldots \vee \alpha^{-n+1}(\mathcal{U}). 
\]
Proposition 1.1.4 of \cite{Coornaert} implies that $\mathcal{D}(\mathcal{V}, \alpha, n) \geq \mathcal{D}(\mathcal{U}, \alpha, n)$ for each $n$. Since $\alpha \in H_{\mathcal{V}, k}$ for each $k \in \N$, 
\[
\lim_{n\rightarrow \infty} \frac{\mathcal{D}(\mathcal{U}, \alpha, n)}{n} \leq \lim_{n\rightarrow \infty} \frac{\mathcal{D}(\mathcal{V}, \alpha, n)}{n} =0.
\]
The choice of $\mathcal{U}$ was arbitrary, so ${\rm mdim}(\alpha)=\sup_{\mathcal{U}} \lim_{n\rightarrow \infty} \frac{\mathcal{D}(\mathcal{U}, \alpha, n)}{n} =0$.
\end{proof}

 Now the main theorem of this section follows from the previous lemma, the proof of \cite[Theorem 1]{GlaWei:MinSkePro}, and the Baire Category Theorem.
\begin{theorem}
There is a residual subset of $\overline{\mathcal{S}(\zeta \times id)}$ consisting of homeomorphisms that are minimal and have mean dimension zero. The same results holds for $\overline{\mathcal{O}(\zeta \times id)}$.
\end{theorem}

\section{$K$-theory and orbit-breaking} \label{Sect:K}

In this section, we establish 
some further results on the $K$-theory of the orbit-breaking subalgebras which will be needed. The definition of orbit-breaking subalgebras was discussed in Section~\ref{Sect:Prelim}. The first lemma and its corollary follow from an inspection of the maps defined in Theorem 2.4 (and Example 2.6) of \cite{Put:K-theoryGroupoids}. 

\begin{lemma} \label{orbitBreakFuncProp}
Let $(X, \varphi)$ be a minimal dynamical system and $Y_1$, $Y_2$ nonempty closed subsets meeting every $\varphi$-orbit at most once. Suppose that $Y_1 \subset Y_2$ and let $j : Y_1 \to Y_2$ denote the inclusion. Let $A := C(X) \rtimes_{\varphi} \mathbb{Z}$. Then $\iota_i : A_{Y_i} \into A$, $i = 1, 2$ and the two exact sequences of Theorem~\ref{PutExtSeq} are compatible. That is, 
\begin{displaymath} 
\xymatrix{ \cdots \ar[r] & K_1(A) \ar@{=}[d] \ar[r] & K^0(Y_2) \ar[d]_j \ar[r] & K_0(A_{Y_2}) \ar[d] \ar[r]^-{(\iota_2)_*} & K_0(A) \ar[r] \ar@{=}[d]  & \cdots\\
\cdots \ar[r] & K_1(A) \ar[r] & K^0(Y_1) \ar[r] & K_0(A_{Y_1}) \ar[r]^-{(\iota_1)_*} & K_0(A) \ar[r] & \cdots .}
\end{displaymath}

\end{lemma}

In the next corollary, (\ref{NCP}) was also observed by Phillips in \cite[Theorem 4.1(3)]{MR2320644}.

\begin{corollary} \label{corOrbitBreakYandPoint}
Applying the previous lemma to the case $Y_1=\{y \}$ for some $y \in X$, $Y_2$  connected and $K_1(A) \cong \Z$, we have 
\begin{displaymath} 
\xymatrix{ \cdots \ar[r] &  \mathbb{Z} \ar@{=}[d] \ar[r] & K^0(Y_2) \cong \mathbb{Z} \oplus \tilde{K}^0(Y_2)  \ar[d]_j \ar[r] & K_0(A_{Y_2}) \ar[d] \ar[r]^-{(\iota_2)_*} & K_0(A) \ar[r] \ar@{=}[d]  & \cdots\\
\cdots \ar[r] & \mathbb{Z} \ar[r] & K^0(\{y\}) \cong \mathbb{Z} \ar[r] & K_0(A_{\{y\}}) \ar[r]^-{(\iota_1)_*} & K_0(A) \ar[r] & \cdots .}
\end{displaymath}

Moreover, 
\begin{enumerate}
\item the map $K_0(A_{\{ y \}}) \rightarrow K_0(A)$ is an isomorphism; \label{NCP}
\item the map $\Z \rightarrow K^0(Y_1)\cong \Z$ is an isomorphism; 
\item the map $K^0(Y_2)  \cong \Z \oplus \tilde{K}^0(Y_2) \rightarrow K^0(\{ pt \}) \cong \Z$ is given by 
\[ (n, y) \in \Z \oplus \tilde{K}^0(Y_2) \mapsto n \in \Z. \]
\end{enumerate}
\end{corollary}

For the next lemma, recall from Section~\ref{Sect:Prelim} that the boundary maps in the Pimsner--Voiculescu exact sequence is denoted $\partial_{\rm PV}$, and the boundary maps in the exact sequence of Theorem~\ref{PutExtSeq} is denoted $\partial_{\rm OB}$.

\begin{lemma} \label{boundaryMapLemma}
Let $(X, \varphi)$ be a minimal dynamical system with $X$ an infinite compact metric space. Suppose $Y \subseteq X$ is a closed subset of $X$ such that $\varphi^{n}(Y) \cap Y = \emptyset$ for every $n \neq 0$. Let $j : Y \into X$ denote the inclusion of $Y$ in $X$. Then the following diagram commutes:
\begin{displaymath} 
\xymatrix{ K_*(A) \ar[d]_{\partial_{\rm OB}} \ar[r]^{\partial_{\rm PV}} & K^{*+1}(X) \ar[dl]^{j^*} \\
K^{*+1}(Y). }
\end{displaymath}
\end{lemma}

 Before beginning the proof of Lemma \ref{boundaryMapLemma}, we need a more explicit description of the boundary maps. The boundary map $\partial_{\rm\, PV}$ of the Pimsner--Voiculescu exact sequence is discussed first.
   
Regard the crossed product as being generated by functions in $C(X)$ and a unitary $u$ such that $ufu^{*} = f \circ \varphi^{-1}$, for $f$ in $C(X)$. Let $b \in \mathcal{B}(\ell^{2}(\mathbb{Z}))$ denote the bilateral shift and let $\mathcal{E} \subset (C(X) \times_{\varphi } \mathbb{Z}) \otimes \mathcal{B}(\ell^{2}(\mathbb{Z}))$ be the $\mathrm{C}^*$-subalgebra generated $C(X) \otimes 1$ and $u \otimes b$.  Denote by $p^+ \in \mathcal{B}(\ell^{2}(\mathbb{Z}))$ the projection onto $\ell(\mathbb{Z}_{>0})$ and put $p := 1_{C(X)} \otimes p^+$. There is a short exact sequence
     \[
         0 \rightarrow C(X) \otimes \mathcal{K}(\ell^{2}(\Z_{>0})) 
  \rightarrow p\mathcal{E}p \rightarrow C(X) \times_{\varphi} \Z \rightarrow 0,
  \]       
  and the map $\partial_{\rm\, PV}$ is the index map in the associated sequence on $K$-theory.

  To describe the second boundary map $\partial_{\text{OB}}$ we elaborate a little on the $K$-theory of the orbit-breaking 
  subalgebras. Let $A := C^{*}(\mathcal{R}_{\alpha})$ and $A_{Y} := C^{*}(\mathcal{R}_{Y})$.

   Let $\mathbb{Z}$ act on 
   $Y \times \mathbb{Z}$, 
   trivially on $Y$ and by translation on $\mathbb{Z}$, and let $\mathcal{S}_{\varphi} $ denote its associated groupoid. We view $\mathcal{S}_\varphi$ as the equivalence relation
   \[
   \mathcal{S}_{\varphi} = 
   \{ \left( (y,m), (y, n) \right) \mid y \in Y, m,n \in \mathbb{Z} \}.
   \]
Define
   \begin{eqnarray*}
   \mathcal{S}_{Y}  & =  &  \mathcal{S} \setminus
   \{ \left( (y,m), (y, n) \right)  \mid y \in Y; m < 1 \leq n 
    \text{ or } n < 1 \leq m 
     \},
   \end{eqnarray*}
and let $C := C^{*}(\mathcal{S}_{\alpha})$ and $C_{Y} := C^{*}(\mathcal{S}_{Y})$ be the associated $\mathrm{C}^*$-algebras.  Observe that by defining
    \begin{eqnarray*}
   \mathcal{S}_{Y}^{+}  & :=  &
   \{ \left( (y,m), (y, n) \right) \mid y \in Y, m,n \geq 1\}, \\
   C_{Y}^{+} & = &  C^{*}\left(\mathcal{S}_{Y}^{+}  \right), \\
    \mathcal{S}_{Y}^{-}  & :=  &
   \{ \left( (y,m), (y, n) \right) \mid y \in Y, m,n < 1\}, \\
   C_{Y}^{-} & := &  C^{*}\left(\mathcal{S}_{Y}^{-}  \right),
   \end{eqnarray*}
  we have  $\mathcal{S}_{Y} = \mathcal{S}_{Y}^{+} \cup \mathcal{S}_{Y}^{-}$ and $C_{Y} =    C_{Y}^{+}  \oplus C_{Y}^{-}$. It is straightforward to check that (see also \cite[Example 2.6]{Put:K-theoryGroupoids}) that for the associated $\mathrm{C}^*$-algebras we have
   \begin{eqnarray*}
   C & \cong & C(Y) \otimes \mathcal{K}\left(\ell^{2}(\mathbb{Z})\right), \\
   C_{Y}^{+} & \cong & C(Y) \otimes \mathcal{K}\left(\ell^{2}(\mathbb{Z}_{>0})\right), \\
   C_{Y}^{-} & \cong & C(Y) \otimes \mathcal{K}\left(\ell^{2}(\mathbb{Z}_{\leq 0})\right).
   \end{eqnarray*}
Note that the inclusion $C_{Y} \subset C$ is considerably more elementary than that of $A_{Y} \subset A$. We want to establish a relation between these two pairs of inclusions.  The map $j: Y \times \mathbb{Z} \rightarrow X$  defined by
   $j(y,n) = \varphi^{n}(y)$ is continuous and  equivariant. By abuse of notation, we will also use $j$ to denote $ j \times j$, which is a  groupoid  morphism from $\mathcal{S}_{\varphi} $ to $\mathcal{R}_{\varphi}$. Observe that $j(\mathcal{S}_{Y} )$ is precisely $j(\mathcal{S}_{\varphi}) \cap \mathcal{R}_{Y}$.
    
    If $g \in C_{c}(\mathcal{R}_{\varphi})$, then $g \circ j: \mathcal{S}_{\varphi} \rightarrow \mathbb{C}$ is continuous and although it does not have compact support, we nevertheless have that for any $f \in C_{c}(\mathcal{S}_{\varphi})$ the convolution products $f  (g \circ j), (g \circ j) f$ are both defined and contained in $C_{c}( \mathcal{S}_{\varphi})$. Thus elements of $A$ act as multipliers on $C$. Similarly, elements of $A_{Y}$ acts as multipliers on $C_{Y}$.  

The main result of \cite{Put:Excision} is that there is an isomorphism between relative $K$-groups, $K_{*}(A_{Y}; A) \cong K_{*}(C_{Y};C)$. Thus we can reduce the computation of $K_{*}(A_{Y}; A)$ to the computation of $K_{*}(C_{Y};C)$. Now, given any $\mathrm{C}^*$-algebra $B$ with $\mathrm{C}^*$-subalgebra $B'$, by \cite[Section 2]{Put:Excision}, the relative $K_0$-group,  $K_0(B', B)$ is shown to consist of  elements  represented by partial isometries in matrices over $\tilde{A}$ whose initial and final projections lie in matrices over $\tilde{A'}$. There is an exact sequence
\begin{displaymath}
\xymatrix{ K_1(B) \ar[r] & K_0(B';B) \ar[r] &K_0(B') \ar[d]  \\
 \ar[u] K_1(B') & \ar[l]  K_1(B';B) & \ar[l] K_0(B) ,
}
\end{displaymath}
given as follows. The map $K_1(B) \to K_0(B';B)$ is defined by considering any unitary in matrices over $\tilde{B}$ as a partial isometry with initial and final projections in matrices over $\tilde{A'}$, and the map $K_0(B'; B) \to K_0(B')$ is defined by sending a partial isometry $v$ to  $[v^{*}v]_{0} - [vv^{*}]_{0} \in K_{0}(B')$. The vertical maps are induced by the inclusion $B'\subset B$.

With the concrete description of the $\mathrm{C}^*$-algebras $C_Y \subset C$, it is a simple matter to check that this yields a short exact sequence
\[
0 \rightarrow K_{0}(C_{Y}; C) \rightarrow K_{0}(C_{Y}^{+})  \oplus K_{0}(  C_{Y}^{-})  \rightarrow
K_{0}(C)  \rightarrow 0,
\]
and that the inclusions $C_{Y}^{+}, C_{Y}^{-} \subset C$ both induce isomorphisms on $K$-theory. Consequently, if $q:= 1 \oplus 0$ in the multiplier algebra of $C_{Y}^{+} \oplus C_{Y}^{-}$,  the map from $K_{0}(C_{Y};C)$ to $K_{0}(C)$ 
taking the class of a partial isometry $v$ in $C$ with initial and final projections in $C_{Y}^{+} \oplus C_{Y}^{-}$ to $[v^{*}vq]_{0} - [vv^{*}q]_{0}$, is an isomorphism.
    
  We now complete our description of the map $\partial_{\text{OB}}$. For each $m \geq 1$, let $e_{m}$ be the characteristic function of the compact open set $\{ (y, i, i) \mid |i| \leq m \} \subset \mathcal{S}_{\alpha}$.  These elements form an approximate unit for $C$ which lies in $C_{Y}$. We use the same $e_{m}$ to denote $e_{m} \otimes 1_{n}$ in $M_{n}(C)$.

Let $v \in M_n(A)$ be a unitary, which we regard as a partial isometry with initial and range projections in $M_{n}(A')$. Define
  \[
  v_{m} := \left[ \begin{array}{cc} ve_{m} & 0 \\ 1_{n}-e_{m} & 0 \end{array} \right],
  \]
  which is a partial isometry in $M_{2n}( \tilde{C} )$ satisfying  
  \[v_{m}^{*}v_{m} = 1 \oplus 0, \qquad v_{m}v_{m}^{*} =  ve_{m}v^{*} \oplus (1-e_{m} ).\]
    For sufficiently large values of $m$, $v_{m}v_{m}^{*}$
     will lie (at least approximately) in $M_{2n}(\tilde{C_{Y}})$. The isomorphism 
     from \cite{Put:Excision} carries the class of $v$ to that of $v_{m}$, for any sufficiently large $m$. Thus for the map $\partial_{OB}$ we have 
     \[ [v]_{1} \mapsto [q(1 \oplus 0 )]_{0} - [ q(ve_{m}v^{*} \oplus (1-e_{m} ))]_{0} \in K_{0}(C).\]
    
\noindent\emph{Proof of Lemma \ref{boundaryMapLemma}}. It will be convenient to let $e^{+}_{m}$ denote
  the characteristic function of the set $\{ (y, i, i) \mid 1 \leq i \leq m\}$. In other words, 
  $e^{+}_{m} = pe_{m}$. First, we show that $j^*\circ\partial_{\text{PV}}([u]_1) = \partial_{\text{OB}}([u]_1)$ where $u$ is the canonical unitary in the crossed product. We compute the Pimsner--Voiculescu map as follows. Lift $u$ to
$p(u \otimes b)p = u \otimes s$ in $p\mathcal{E}p$, where $s$ is the unilateral shift.
Then
\[
  \partial_{\text{PV}}([u]_{1}) = [1 \otimes ss^{*}]_{0} - [ 1 \otimes s^{*}s]_{0} = - [1 \otimes (I-ss^{*})]_{0}. 
  \]
  Under the isomorphism of $K_{0}(C(X) \otimes \mathcal{K})$ with $K_{0}(C(X))$, this is identified with
  $- [1]_{0}$. On the other hand,  
  \begin{eqnarray*}
  \partial_{\text{OB}}[u]_{1} & =  & [q(1 \oplus 0)]_{0} - [ q(ue_{m}u^{*} \oplus (1-e_{m} ))]_{0} \\
      &  =  &  
  [q]_{0} -  [ e^{+}_{m+1} \oplus (q-e^{+}_{m} ))]_{0} \\
    & = & -[e_{1}^{+}]_{0}.
  \end{eqnarray*}
  Under the isomorphism between $K_{0}(C)$ and $K_{0}(C(Y))$, 
 $[e_{1}^{+}]_{0}$ is identified with
  $[1]_{0}$. As the restriction map from $C(X)$ to $C(Y)$ is unital, we have $j^*\circ\partial_{\text{PV}}([u]_1)) = \partial_{\text{OB}}([u]_1)$, as required.
    
Now we show that $j^*\circ\partial_{\text{PV}} = \partial_{\text{OB}}$ when applied to an arbitrary unitary in  $M_{n}(\tilde{A})$. For simplicity, assume $n=1$. We may approximate this unitary
  by an invertible $v = \sum_{-N}^{N} f_{n}u^{n}$, where $N \in \mathbb{N}$ and $f_{n} \in C(X)$,  $-N \leq n \leq N$.
  As we know the conclusion holds for $u$, it suffices to prove it holds for
  $u^{N}v = \sum_{0}^{2N} f_{n-N}u^{n}$.
  
  For $m > 2N$, the initial and final projections of $(u^{N}v)_{m}$ are in $\tilde{C_{Y}}$
  and  we have
  \[
  \partial_{\text{OB}}[u^{N}v] = [qu^{N}ve_{m}v^{*}u^{N}]_{0} - [e_{m}^{+}]_{0}.
  \]
  On the other hand, we may lift the element $u^{N}v = \sum_{n=0}^{2N} f_{n-N}u^{n}$ to the element $w = \sum_{n=0}^{2N} f_{n}u^{n} \otimes B^{n}$ in $\mathcal{E}$. From the fact that $p(u \otimes b)p=(u \otimes b)p$, it follows that  $pwp=wp$ and is a partial isometry (approximately) with $(wp)^{*}(wp) \approx p$ and  $(wp)(wp)^{*} = wpw^{*} = pwpw^{*}p$, which is a subprojection of $p$. It follows that 
 \[
 \partial_{\text{PV}}[u^{N}v]_{1} = -[p - wpw^{*}]_{0}.
 \]
  
  We also know that 
   \[
   p(u \otimes b)^{*}p = p(u \otimes b)^{*}
   \]
   and hence, provided $m > 1$, 
   \[
   pw(u \otimes b)^{*m}p = pw(u \otimes b)^{*m}.
   \]
If we let $d_{m} = 1 \otimes p_{\ell^{2}\{-m, \ldots, m \}}$ where $ p_{\ell^{2}\{-m, \ldots, m \}}$ is the projection
   onto \newline $\ell^{2}\{-m, \ldots, m \}$, we have 
  \[
   (1-p)d_{m} =   (1-p)(u \otimes b)^{*(m+1)}p (u \otimes b)^{m+1}.
  \]
It is clear that  $wp(1-d_{m})w^{*}$
  is a subprojection of $wpw^{*}$ and hence also of $p$. It follows that 
  \[
 [p-wpw^{*}]_{0} = [p- wp(1-d_{m})w^{*}]_{0} - [wpw^{*}- wp(1-d_{m})w^{*}]_{0}
  \]
  
  For the second term, we have 
 \[
 [wpw^{*}- wp(1-d_{m})w^{*}]_{0} = [wpd_{m}w^{*}]_{0} = [pd_{m}]_{0}.
  \]
  and for the first, we have
 
   \begin{eqnarray*}
   p - wp(1-d_{m})w^{*} & = & p - wpw^{*} + wpd_{m}w^{*} \\
                  & = &  p - pwpw^{*}p + wpd_{m}w^{*} \\
                   & = &  pw(1-p)w^{*}p + wpd_{m}w^{*} \\
                     &  =  & 
            pw(u \otimes b)^{*(m+1)} (u \otimes b)^{m+1} (1-p)w^{*}p \\
   &  &            +   wpd_{m}w^{*} \\ 
        &  =  &  
          pvw(u \otimes b)^{*(m+1)}p (u \otimes b)^{m+1} (1-p)w^{*}p  \\
   &  &       + wpd_{m}w^{*}  \\
     &  =  &  
          pwd_{m} (1-p)w^{*}p 
     + wpd_{m}w^{*}  \\
     & = & pwe_{m}w^{*}p.
            \end{eqnarray*}

  Applying the restriction map from $C(X)$ to $C(Y)$ takes $p$ to $q$ and 
  $d_{m}$ to $e_{m}$, and we are done.
 \qed

\section{$\mathrm{C}^{*}$-algebras associated to skew products systems} \label{Sect:Skew}

Let $(X, \alpha)$ be a minimal dynamical system as constructed by Theorem \ref{GlasnerWeissToZ}. By Proposition~\ref{alphaActZeroKtheory},  the homeomorphism $\alpha$ induces the identity on $K$-theory, so the Pimsner--Voiculescu exact sequence is given by
\begin{displaymath} 
\xymatrix{
K^0(X) \ar[r]^-{0} & K^0(X) \ar[r] &  K_0(C(X) \rtimes_{\alpha} \Z)  \ar[d]^{\partial_{\rm PV}}\\
 K_1(C(X) \rtimes_{\varphi} \Z) \ar[u]^{\partial_{\rm PV}} &  K^1(X) \ar[l] & K^1(X) \ar[l]_-{0},}
\end{displaymath}
and hence reduces to the  two short exact sequences
\begin{align*}
0 \rightarrow K^0(X) \rightarrow K_0(C(X) \rtimes_{\alpha} \Z) \rightarrow K^1(X) \rightarrow 0, 
\end{align*}
and
\begin{align*}
0 \rightarrow K^1(X) \rightarrow K_1(C(X) \rtimes_{\alpha} \Z) \rightarrow K^0(X) \rightarrow 0.
\end{align*}
In general we do not know if these short exact sequences split (compare with \cite[Proposition 4.2 and Remark 4.3]{Suk:ConMinSys}). However, there exist minimal skew products where these sequences do indeed split, and in these cases $K$-theory computations are possible.

As before, let $W$ be a finite connected CW-complex and let $Q$ be the Hilbert cube. As in Subsection~\ref{minSkewProSec}, we apply \cite[Theorem 1]{GlaWei:MinSkePro} to $S^d\times W \times Q$ to obtain a minimal homeomorphism 
\[
\tilde{\varphi} : S^d\times W \times Q \rightarrow S^d\times W \times Q, \qquad
(s, w) \mapsto (\varphi(s), h_s (w)),
\]
where $s \in S^d$, $w\in W\times Q$ and $h : S^d \rightarrow {\rm Homeo}(W\times Q)$. Let  $(Z\times W\times Q, \,  \tilde{\zeta})$ be the corresponding minimal system constructed as an extension of $(S^d\times W \times Q, \, \tilde{\varphi})$ as in Proposition~\ref{MinMapConTwo}.

\begin{proposition} \label{KtheoryCrossed}
Let $X :=  Z\times W\times Q$. Then for the minimal dynamical system $(X, \tilde{\zeta})$, we have 
\begin{align*}
K_0(C(X) \rtimes_{\tilde{\zeta}} \Z) & \cong K_1(C(X) \rtimes_{\tilde{\zeta}}\Z ) \cong K^0(X)\oplus K^1(X).
\end{align*}
\end{proposition}
\begin{proof} 
The result will follow by showing the short exact sequences obtain from the Pimsner--Voiculescu exact sequence split. In turn, by \cite[Proposition 10.5.1]{Bla:K-theory}, this will follow by showing that $\tilde{\zeta}$ is homotopic to a conjugate of $\zeta\times \id_W \times \id_Q$. The space $S^d\times W \times Q$ is a compact Hilbert cube manifold and hence its homeomorphism group is locally contractible by the main result of \cite{Cha:HomHibCubMfd}. Using this fact together with the skew product construction, there exists a homotopy
\[ H : S^d \times [0,1] \rightarrow {\rm Homeo}(W\times Q), \]
where 
\[ H(s, 0)=h_s \hbox{ and } H(s, 1) = g^{-1}_{\varphi(s)} \circ g_s, \]
for some $g: S^d \rightarrow {\rm Homeo}(W\times Q)$. Precomposing with the factor map $\tilde{q}$ leads to
\[ \tilde{H}: Z\times  [0,1] \rightarrow {\rm Homeo}(W\times Q), \]
where 
\[ \tilde{H}(z, 0)=h_{q(s)} \hbox{ and } \tilde{H}(s, 1) = g^{-1}_{\varphi(q(z))} \circ g_{q(z)} .\]
Summarizing, we have obtained a homotopy from $\tilde{\zeta}$ to the homeomorphism
\[
(z, w) \mapsto (\zeta(z), g^{-1}_{\varphi(q(z))} \circ g_{q(z)}(w) )= (G^{-1} \circ (\zeta \times id_W \times id_Q) \circ G)(z,w),
\]
where the homeomorphism $G$ is determined by $g$ via
\[
G(z,w)=(z, g_{q(z)}(w)).
\]
The result follows.
\end{proof}

Although the $K$-theory of the crossed product algebra is not the same as the space $W$, the orbit-breaking construction allows us to obtain a minimal equivalence relation whose $\mathrm{C}^*$-algebra has the same $K$-theory as the space $W$. We note that this is the case both for the minimal homeomorphisms constructed in Theorem \ref{GlasnerWeissToZ} and Proposition~\ref{MinMapConTwo} but we will concentrate on the former.

Let $A := C(X) \times_{\alpha} \mathbb{Z}$ where $X :=Z\times W\times Q$ and $\alpha$ is minimal and in $\overline{\mathcal{S}(\zeta \times id)}$. Then for any $z_0 \in Z$, the set $\tilde{W}=\{ z_0 \} \times W\times Q$ is a closed subset of $X$. Let $A_{\tilde{W}}$ denoted the orbit-breaking subalgebra associated to $\tilde{W}$, that is, $A_{\tilde{W}} = \mathrm{C}^*(\mathcal{R}_{\tilde{W}})$ where $\mathcal{R}_{\tilde{W}}$ is the minimal equivalence relation defined in Section~\ref{Sect:Prelim}.

\begin{theorem} \label{OBK}
The $\mathrm{C}^*$-algebra $A_{\tilde{W}}$ is simple, separable, unital and nuclear. Moreover, if $A := C(X) \times_{\alpha} \mathbb{Z}$ is $\mathcal{Z}$-stable, then so is $A_{\tilde{W}}$. In particular, this is the case when $(X, \alpha)$ has mean dimension zero.
\end{theorem}
\begin{proof}
The nonempty closed subset $\tilde{W}$ meets each $\alpha$-orbit once since $\alpha(z, w, q)=(\zeta(z), h_z(w,q))$ and $\zeta$ is minimal. Thus the statement follows from Theorem ~\ref{classifiableTheorem}.
\end{proof}
\begin{theorem}
The $K$-theory of the orbit-breaking algebra $A_{\tilde{W}}$ is isomorphic to the $K$-theory of $W$, that is,
\[K_*(A_{\tilde{W}}) \cong K^*(W).\]
\end{theorem} 
\begin{proof}
The inclusion map $\tilde{W} \hookrightarrow X$ induces an isomorphism on $K$-theory and moreover $K^*(\tilde{W}) \cong K^*(W)$. Hence the six-term exact sequence of Theorem~\ref{PutExtSeq} becomes
\begin{displaymath}
\xymatrix{ K^0(W) \ar[r] & K_0(A_{\tilde{W}}) \ar[r] & K_0(A) \ar[d]\\
K_1(A) \ar[u] & K_1(A_{\tilde{W}}) \ar[l] & K^1(W). \ar[l]}
\end{displaymath}

Using Lemma \ref{boundaryMapLemma} and again the fact that the inclusion map $\tilde{W} \hookrightarrow X$ induces an isomorphism on $K$-theory, the boundary maps in this sequence are onto and hence the six-term exact reduces to two short exact sequences:
\[
0 \rightarrow K_*(A_{\tilde{W}}) \rightarrow K_*(A) \rightarrow K^{*+1}(\tilde{W}) \rightarrow 0 ,
\]
for $i = 0, 1$.  Lemma \ref{boundaryMapLemma}, together with the commutative diagram
 \begin{displaymath} 
\xymatrix{ & C(X) \ar[dl]_{i_1} \ar[dr]^{i_2} &\\
A_{\tilde{W}} \ar[rr]_{\iota} & & A,}
\end{displaymath} 
 implies that the diagram 
\begin{displaymath}
\xymatrix{ 0  \ar[r] & K_*(A_{\tilde{W}})  \ar[r] & K_0(A) \ar[r]  & K^{*+1}(\tilde{W}) \ar[r]  & 0 \\
0  \ar[r] & K^*(X) \ar[r] \ar[u]  & K_0(A) \ar[r] \ar@{=}[u]& K^{*+1}(X) \ar[r] \ar[u] & 0}
\end{displaymath}
commutes, where the second exact sequence is obtained from the Pimsner--Voiculescu exact sequence via the discussion following Proposition \ref{alphaActZeroKtheory}. Since the map $K^{*+1}(X) \rightarrow K^{*+1}(\tilde{W})$ is an isomorphism, the Five Lemma implies that $K_*(A_{\tilde{W}}) \cong K^*(X) \cong K^*(W)$.
\end{proof}

Note that the theorem above does not tell us anything about the order structure on the $K_0$-group. However,  in particular cases we can obtain a complete description. For any given a finite CW-complex, there exists a connected CW-complex with the same $K$-theory but with trivial first cohomology group.

\begin{theorem}
Let $X = Z \times W \times Q$ and suppose $W$ has $H^{1}(W) = 0$. Then $A := C(X) \rtimes_{\alpha} \mathbb{Z}$ has no nontrivial projections. Moreover, for any tracial state $\tau \in T(A)$ the image of the map
\[ \tau_* : K_{0}(A) \to \mathbb{R}, \qquad  [p] \to \tau(p),\] is $\mathbb{Z}$ and $x \in K_0(A)_+$ if and only if $\tau(x) \geq 0$.
\end{theorem}

\begin{proof}
 This follows from a theorem of Connes \cite{MR605351} (see also \cite[Corollary 10.10.6]{Bla:K-theory}) In this case we can also compute the order structure on $K_0$ since the range of the trace on $K_0(C(X)\rtimes_{\alpha}\Z)$ is $\Z$ (see the proof of \cite[Corollary 10.10.6]{Bla:K-theory}).
\end{proof}

Let $A$ be a simple, nuclear $\mathrm{C}^*$-algebra and $\tau \in T(A)$ a tracial state. For any positive element, set $d_\tau(a) := \lim_{n \to \infty} \tau(a^{1/n})$. We say that a $\mathrm{C}^*$-algebra has \emph{strict comparison} (of positive elements) if, whenever $a, b \in A$ are positive elements such that $d_\tau(a) < d_\tau(b)$ for every $\tau \in A$, then there exists a sequence $(r_n)_n \in M_{\infty(A)}$ such that $\| r_n b r_n - a \| \to 0$ as $n \to \infty$. If $A$ has strict comparison and $p$ and $q$ are projections in $M_{\infty}(A)$, then $\tau(p) < \tau(q)$ for every $\tau \in T(A)$ implies that $[p] < [q] \in K_0(A)$.

The previous theorem allows us to say something about the order structure on the orbit-breaking subalgebra

\begin{corollary}
Let $X = Z \times W \times Q$ and suppose $W$ has $H^{1}(W) = 0$.  Suppose that $(X, \alpha)$ has mean dimension zero. Let $\tilde{W} = \{ z_0 \} \times W \times Q$. Then the orbit-breaking subalgebra $A_{\tilde{W}} \subset  A:= C(X) \rtimes_{\alpha} \mathbb{Z}$ has no nontrivial projections and $K_0(A_{\tilde{W}})_+ \cong \iota_*(K_0(A_{\tilde{W}}) \cap K_0(A)_+$, where  $\iota : A_{\tilde{W}} \to A$ is the inclusion map.
\end{corollary}

\begin{proof}
 That $A_{\tilde{W}}$ has no nontrivial projections is immediate. From the proof of Theorem~\ref{OBK} we see that $\iota_* : K_0(A_{\tilde{W}}) \to K_0(A)$ is injective. Suppose that we have $x \in K_0(A)_+$ satisfying $x \geq 0$. Then there are projections $p, q \in M_{\infty}(A_{\tilde{W}})$ such that $\iota_*^{-1}(x) = [p]-[q] \in K_0(A_{\tilde{W}})$. Since $(X, \alpha)$ has mean dimension zero, $A_{\tilde{W}}$ is $\mathcal{Z}$-stable and hence has strict comparison \cite{Rordam2004}. Thus it is enough to show that $\tau(p) > \tau(q)$ for any $\tau \in T(A_{\tilde{W}})$, since this in turn implies $[p]- [q] > 0$ and therefore $y \in K_0(A_{\tilde{W}})_+$ if and only if $\iota_*(y) \in K_0(A)$. Indeed, since any $\tau \in T(A_{\tilde{W}})$ is of the form $\tau = \sigma \iota$ for some $\sigma \in T(A)$, we have $0 < \sigma(x) = \sigma(\iota_*([p] - [q])) = \tau(p) - \tau(q)$. Thus $[p] - [q] \in K_0(A_{\tilde{W}})_+$ and the result follows.
\end{proof}

By using the previous corollary and Example \ref{uniqueErgodic}, we can obtain monotracial crossed products with only  trivial projections where the relevant space $X$ satisfies $K^0(X) \cong \Z \oplus G_0$ and $K^1(X)\cong G_1$. Furthermore, using Proposition \ref{KtheoryCrossed}, the minimal dynamical system can be taken so that the $K$-theory of the crossed product is $K_0(C(X)\rtimes \Z) \cong \Z \oplus G_0 \oplus G_1$ and $K_1(C(X) \rtimes \Z ) \cong \Z\oplus G_0 \oplus G_1$.

\section{Applications to the Elliott program} \label{Sect:Class}

Let $A$ be a simple, separable, unital, nuclear $\mathrm{C}^*$-algebra. The \emph{Elliott invariant} of $A$, denoted $\Ell(A)$, is the 6-tuple
\[ \Ell(A) := (K_0(A), K_0(A)_+, [1_A], K_1(A), T(A), r_A : T(A) \to S(K_0(A))),\]
where $r_A : T(A) \to S(K_0(A))$ maps a tracial state $\tau$ to the state on the ordered abelian group $(K_0(A), K_0(A)_+,  [1_A])$ defined by $(\tau)([p]-[q]) =  \tau(p) - \tau(q)$, for projections $p, q \in M_{\infty}(A)$.

In this section, we use the results of the previous sections to partially address Question~\ref{q2} of the introduction: When is a classifiable $\mathrm{C}^*$-algebras isomorphic to the $\mathrm{C}^*$-algebra of a minimal \'etale equivalence relation?  The classification theorem, stated here, is the culmination of many years of work. 

\begin{theorem}[see \cite{CETWW, BBSTWW:2Col, EllGonLinNiu:ClaFinDecRan, GongLinNiue:ZClass,  TWW}]
 \label{ClassThm} Let $A$ and $B$ be separable, unital, simple, infinite-dimensional \mbox{$\mathrm{C}^*$-algebras} with finite nuclear dimension and which satisfy the UCT. Suppose there is an isomorphism 
\[ \varphi : \Ell(A) \to \Ell(B).\]
Then there is a $^*$-isomorphism 
\[ \Phi : A \to B,\]
which is unique up to approximate unitary equivalence and satisfies $\Ell(\Phi) = \phi$.
\end{theorem}

We  denote the class of all  unital, simple, infinite-dimensional \mbox{$\mathrm{C}^*$-algebras} with finite nuclear dimension and which satisfy the UCT by $\mathcal{C}$, that is,
 \[ \mathcal{C} := \{ A \mid A \text{ a classifiable $\mathrm{C}^*$-algebra}\}.\]

\subsection{Projectionless $\mathrm{C}^*$-algebras} \label{Subsect:NoProj}

In Section~\ref{Sect:Skew} the space $W$ was a finite CW complex and consequently all the $\mathrm{C}^*$-algebras considered have finitely generated $K$-theory. 
 In this subsection, orbit-breaking subalgebras without this restriction are constructed by starting from minimal dynamical systems on point-like spaces. The resulting $\mathrm{C}^*$-algebras will only have trivial projections $0$ and $1$. Since the point-like spaces have finite covering dimension, the minimal dynamical systems will have mean dimension zero and so the $\mathrm{C}^*$-algebras of this section will all be $\mathcal{Z}$-stable and hence belong to $\mathcal{C}$. Furthermore, by tensoring with a full matrix algebra we can also vary the class of the unit in the Elliott invariant, see Corollary \ref{cor:fewProjections}.
 
Let $G_0$ and $G_1$ be arbitrary countable abelian groups. Standard results imply that we can take a compact connected metric space $Y$ with $\dim(Y) < \infty$ and
\[ K^0(Y) \cong \mathbb{Z} \oplus G_0, \qquad K^1(Y)  \cong G_1.\]

We now consider $Y$ fixed for the rest of our discussion in this section. Let $d$ be an odd number large enough such that there exists embedding $Y \hookrightarrow S^{d-2}$. Let $(Z, \zeta)$ be a minimal dynamical system constructed from a minimal diffeomorphism $\varphi : S^d \rightarrow S^d$ as in Subsection~\ref{constructionZ}.

\begin{lemma}
There exists an embedding $\iota : Y \to Z$ such that $\varphi^n(\iota(Y)) \cap \iota(Y) = \emptyset$ for every $n \in \mathbb{N} \setminus \{0\}$.  
\end{lemma}

\begin{proof}
In the proof Lemma 1.13 in \cite{DPS:DynZ} there is an embedding of $S^{d-2} \rightarrow Z$ whose image lies in the closed set
$\pi^{-1}\{ x \}$, for some  $x$ in $X$. Compose this with the embedding of $Y$ into the sphere $S^{d-2}$ to get an embedding of $Y$ into $Z$. If $n \neq 0$,  we have $\pi( \zeta^{n} (Y)) = \varphi^{n}(x) \neq x = \pi(Y)$, which implies that $\zeta^{n}(Y) \cap Y $ is empty.
\end{proof}

Let $A := C(Z) \times_{\zeta} \mathbb{Z}$ and let $A_Y$ denote the orbit-breaking subalgebra of $A$.

\begin{theorem} \label{K(A_Y)}
The $\mathrm{C}^*$-algebra $A_Y$ satisfies the following
\[K_0(A_Y) \cong K^0(Y) \cong \mathbb{Z} \oplus G_{0},\qquad K_1(A_Y) \cong K^1(Y) \cong G_{1},\]
and the positive cone of $K_0(A_Y)$ is given by
\[ K_0(A_Y)_{+} \cong \{ (n, z) \mid n=0, z=0, \text{ or }n > 0 \}.\]
\end{theorem}

\begin{proof}
It follows from \cite{DPS:DynZ} that $K_0(A) \cong \mathbb{Z}$ and $K_1(A) \cong \mathbb{Z}$. We have an exact sequence
\begin{displaymath} 
\xymatrix{ & C(Z) \ar[dl]_{i_1} \ar[dr]^{i_2} &\\
A_Y \ar[rr]_{\iota} & & A}
\end{displaymath}
where  the map $i_2$ induces a map
\[ \mathbb{Z} \cong K_0(C(Z)) \to K_0(A) \cong \mathbb{Z}.\]
By the Pimsner--Voiculescu exact sequence, as calculated in the proof of \cite[Proposition 2.8]{DPS:DynZ},  the above map is an isomorphism.  Consequently, the map $\iota_* : K_0(A_Y) \to K_0(A)$, as given in the six-term exact sequence of Theorem~\ref{PutExtSeq}, is onto. Thus the sequence of Theorem~\ref{PutExtSeq} becomes 
\begin{displaymath}
\xymatrix{ \mathbb{Z} \oplus G_0  \ar[r] & K_0(A_Y) \ar[r]^{\iota_*} & \mathbb{Z} \ar[d]^0\\
\mathbb{Z}  \ar[u]^L & K_1(A_Y) \ar[l] & G_1. \ar[l]}
\end{displaymath}
Furthermore, the map $K_0(A_Y) \to \mathbb{Z}$ splits. To see this, note that $i_2$ induces an isomorphism on $K_0$. Thus, using the commutativity of the first diagram in the proof, the splitting map is  given by $(i_1)_* \circ (i_2)_*^{-1}$. 

To complete the proof, we need to show that $L : \mathbb{Z} \to \mathbb{Z} \oplus G_0$ is the map $n \mapsto (n, 0)$. However, this follows from Corollary \ref{corOrbitBreakYandPoint}.
\end{proof}

\begin{corollary} \label{cor:fewProjections} Let $G_0$ and $G_1$ be countable abelian groups and let $\Delta$ be a finite-dimensional Choquet simplex. Then $(\Z \oplus G_0, \Z_{\geq 0}, [1]=k)$ is an ordered abelian group, and if there is a map 
\[ r : \Delta \to S(\mathbb{Z} \oplus G_0), \qquad \tau \mapsto \left((n,g) \mapsto \frac{n}{k}\right) ,\]
then there exists an amenable minimal equivalence relation $\mathcal{R}$ such that 
\[ \Ell(\mathrm{C}^*(\mathcal{R})) \cong (\Z \oplus G_0, \Z_{\geq 0}, [1]=k, G_1, \Delta,  r).\]
In particular if $A \in \mathcal{C}$ and $\Ell(A) =    (\Z \oplus G_0, \Z_{\geq 0}, [1]=k, G_1, \Delta,  r )$, then $A$ and  $\mathrm{C}^*(\mathcal{R})$ are $^*$-isomorphic.
\end{corollary}

\begin{proof}
Let $Z$ be a point-like space and let $(Z, \zeta)$ be a minimal dynamical system with simplex of $\zeta$-invariant measures given by $\Delta$ (which exists by \cite{DPS:DynZ}). By Theorem~\ref{K(A_Y)}, there is  a nonempty closed subset $Y \subset Z$ meeting every $\zeta$-orbit at most once such that the associated minimal equivalence relation  $\mathcal{R}_Y \subset Z \times Z$ has $K$-theory given by $(K_0(\mathrm{C}^*(\mathcal{R}_Y)), K_0(\mathrm{C}^*(\mathcal{R}_Y)) = ( \Z \oplus G_0, \, \mathbb{Z}_+)$ and $K_1(\mathrm{C}^*(\mathcal{R}_Y)) = G_1$.  We have $[1] = 1 \in \mathbb{Z}$. To arrange that $[1] = k \in \mathbb{Z}_{> 0}$, we replace $\mathcal{R}_Y \subset Z \times Z$ with the equivalence relation on $\mathcal{R} \subset (Z \times \{1, \dots, k\}) \times (Z \times \{1, \dots, k\})$ which gives us $\mathrm{C}^*(\mathcal{R}) \cong \mathrm{C}^*(\mathcal{R}_Y)\otimes M_k$.  Since $\iota : \mathrm{C}^*(\mathcal{R}_Y) \into C(Z) \rtimes_{\zeta} \mathbb{Z}$ induces a homoemorphism of tracial state spaces, 
\[ T(\mathrm{C}^*(\mathcal{R})) \cong T(\mathrm{C}^*(\mathcal{R}_Y)) \cong T(C(Z) \rtimes_{\zeta} \mathbb{Z}) \cong \Delta,\]
and since 
\[ r_{\mathrm{C}^*(\mathcal{R}_Y)} :    T(\mathrm{C}^*(\mathcal{R}_Y)) \to S(K_0(\mathrm{C}^*(\mathcal{R}_Y)), \qquad (n, g) \mapsto n \in \mathbb{Z},\]
we have
\[ r_{\mathrm{C}^*(\mathcal{R})} :    T(\mathrm{C}^*(\mathcal{R}_Y)) \to S(K_0(\mathrm{C}^*(\mathcal{R}_Y))), \qquad (n, g) \mapsto \frac{n}{k}.\]
Thus \[ \Ell(\mathrm{C}^*(\mathcal{R})) \cong (\Z \oplus G_0, \Z^+, [1]=k, G_1, \Delta,  r),\] and since $\mathrm{C}^*(\mathcal{R})$ is a simple, separable, nuclear, unital, $\mathcal{Z}$-stable $\mathrm{C}^*$-algebra, Theorem~\ref{ClassThm} implies that for any $A \in \mathcal{C}$ we have $A \cong \mathrm{C}^*(\mathcal{R})$.
\end{proof}
\begin{remark}
It is worth noting taking the class of the unit to be 1 in the previous corollary results in $\mathrm{C}^*$-algebras with only the trivial projections 0 and 1.
\end{remark}

\subsection{$\mathrm{C}^*$-algebras with real rank zero} \label{Subsect:RR0}

In this section, we consider the crossed products and orbit-breaking subalgebras arising form the construction of minimal homeomorphisms on nonhomogeneous spaces as in Subsection~\ref{NonHomSec}. In contrast to the $\mathrm{C}^*$-algebras of the previous subsection, these $\mathrm{C}^*$-algebras will always have a plentiful supply of projections: they all have real rank zero (and this can be read directly from the $K$-theory).

To begin, we require a lemma about the $K$-theory of the crossed product $\mathrm{C}^*$-algebra associated to a nonhomogeneous extension of a Cantor minimal system.  Let $K$ denote the Cantor set and suppose $\varphi : K \to K$ be a minimal homeomorphism and let $(\tilde{K}, \tilde{\varphi})$ be a minimal extension of $(K, \varphi)$ with factor map $\pi : \tilde{K} \to K$, as in Theorem~\ref{FloGjeJohSysThm}. The proof for the case that $\pi^{-1}(x)$, $x \in \tilde{K}$, is either a single point or $[0,1]^n$ is similar to the proof of \cite[Theorem 4]{FloGjeJohSys} (where $n = 1$). See also the proof in \cite{DPS:nonHom}, where $\pi^{-1}(x)$ can be higher dimensional cubes as well as more complicated spaces.  Note that in \cite{FloGjeJohSys} notation  $K(K, \varphi)$ is used to denote the $5$-tuple $(K_0(A), K_0(A)_+, [1_A], K_1(A), T(A))$ for $A = C(K) \rtimes_{\varphi} \mathbb{Z}$.

\begin{lemma} \label{KthFGJsys}
Let $A := C(K) \rtimes_{\varphi} \mathbb{Z}$ and $B := C(\tilde{K}) \rtimes_{\tilde{\varphi}} \mathbb{Z}$. Then the factor map $\pi : (\tilde{K}, \tilde{\varphi})  \to (K, \varphi)$ induces isomorphisms
\[ (K_0(A), K_0(A)_+, [1_A]) \cong  (K_0(B), K_0(B)_+, [1_B]), \qquad K_1(A) \cong K_1(B),\]
and an affine homeomorphism
\[ T(B) \cong T(A).\]
\end{lemma}

Let $G_0$ be a simple dimension group, $T$ a countable abelian torsion group and $G_1$ a countable abelian group. By standard results in topological $K$-theory there exists a compact connected metric space $Y$ with $\dim(Y) < \infty$ and
\[ K^0(Y) \cong \mathbb{Z} \oplus T, \qquad K^1(Y)  \cong G_1.\]
By \cite[Corollary 6.3]{MR1194074}, there exists Cantor minimal system $(K, \psi)$ such that
\[
K_0(C(K) \rtimes_{\psi} \Z) \cong G_0, \qquad K_1(C(K)\rtimes_{\psi} \Z) \cong \Z.
\]
Let $n \in \N$ be large enough so that there exists embedding $Y \hookrightarrow [0,1]^n$ and let  $(\tilde{K}, \tilde{\psi})$ be the nonhomogeneous extension with $\pi^{-1}(x)$ a point or $\pi^{-1}(x) = [0,1]^n$ and which has the same $K$-theory as $(K, \psi)$. By Lemma \ref{KthFGJsys}, there is an explicit isomorphism between the $K$-theory of the crossed product $\mathrm{C}^*$-algebras, induced from the factor map $\pi : \tilde{K} \rightarrow K$, which preserves both the order and the class of the unit.
Let $B :=C(\tilde{K}) \rtimes_{\psi}\Z$ and form $B_Y$, the orbit-breaking subalgebra associated to $Y$, considered as a closed subspace of $\tilde{K}$ via $Y \hookrightarrow I^n \hookrightarrow \tilde{K}$. Since $I^n \cap \psi^k(I^n) = \emptyset$ for $k \neq 0$, we have that $Y\cap \psi^k(Y) = \emptyset$ for $k \neq 0$ and hence  $B_Y$ is simple.

\begin{theorem}
The $K$-theory of the orbit-breaking subalgebra $B_Y$ is given by
\[
(K_0(B_Y), K_0(B_Y)_+, [1]) \cong (T \oplus G_0, G_0^+, 1_{G_0}), \qquad K_1(B_Y) \cong G_1.
\]
\end{theorem}

\begin{proof}
The proof is similar to the proof of Theorem 4.1 in \cite{Putnam:MinHomCantor}. As before, we have the six-term exact sequence of Theorem~\ref{PutExtSeq} and the commutative diagram
\begin{displaymath} 
\xymatrix{ & C(\tilde{K}) \ar[dl]_{i_1} \ar[dr]^{i_2} &\\
B_Y \ar[rr]_{\iota} & & B,}
\end{displaymath}
which gives us the six-term exact sequence
\begin{displaymath}
\xymatrix{ \mathbb{Z} \oplus T  \ar[r] & K_0(B_Y) \ar[r]^{\iota_*} & G_0 \ar[d]^0\\
\mathbb{Z}  \ar[u]^L & K_1(B_Y) \ar[l] & G_1. \ar[l]}
\end{displaymath}
Using Corollary \ref{corOrbitBreakYandPoint}, it follows that
\[
0 \rightarrow \mathbb{Z} \rightarrow \Z \oplus T  \rightarrow K_0(B_Y) \rightarrow G_0 \rightarrow 0, 
\]
and $K_1(B_Y)\cong G_1 $. Moreover, the map $\Z \rightarrow \Z \oplus T$ is given by $l \mapsto (l, 0)$, so we have the short exact sequence
\[
0 \rightarrow T \rightarrow K_0(B_Y) \rightarrow G_0 \rightarrow 0.
\]
To complete the first part of the proof we show that this sequence splits. To do so, consider the maps $Y \rightarrow I^n$ and associated orbit-breaking subalgebras. Lemma \ref{orbitBreakFuncProp} implies that we have 
\begin{displaymath}
\xymatrix{ \ar[r] & \Z \ar[r] \ar[d] & K_0(B_{I^n}) \ar[r] \ar[d] & K_0(C(\tilde{K})\rtimes \Z) \ar[r] \ar@{=}[d] &   \\
\ar[r] & \Z \oplus T \ar[r] &  K_0(B_Y)\ar[r]& K_0(C(\tilde{K})\rtimes \Z) \ar[r] &. }
\end{displaymath}
Moreover, the map $K_0(B_{I^n}) \rightarrow K_0(C(\tilde{K})\rtimes \Z) \cong G_0$ is isomorphism, which gives us the required splitting.

Next we show that the positive cone of $K_0(B_Y)$ is $(G_0)_+$. We follow the proof of Theorem 4.1 in \cite{Putnam:MinHomCantor}. The result will follow by showing that given $a\in (G_0)_+$ there exists $b\in K_0(B_Y)_+$ such that $\iota_*(b)=a$. As such, let $a\in (G_0)_+$. Observe that for the diagram
\begin{displaymath} 
\xymatrix{ & C(\tilde{K}) \ar[dl]_{i_1} \ar[dr]^{i_2} &\\
B_Y \ar[rr]_{\iota} & & B}
\end{displaymath}
there exists $c\in K^0(\tilde{K})_+$ such that $(i_2)_*(c)=a$. Then the element $b=(i_1)_*(c)$ has the required property.

Finally, the class of the unit is respected by the map $\iota_*$ because $\iota$ is unital.
\end{proof}

\begin{corollary}
Let $A$ be a simple, separable, unital $\mathrm{C}^*$-algebra with finite decomposition rank, real rank zero and that satisfies the UCT. Suppose that 
\[ K_0(A) \cong T \oplus G_0, \qquad K_1(A) \cong G_1, \] 
where $T$ is a countable torsion abelian group, $G_0$ is a simple dimension group, $G_1$ is a countable abelian group and the order structure and class of the unit of $K_0(A)$ are the same as the simple dimension group $G_0$.
Then there exists an amenable equivalence relation, $\mathcal{R}$, with 
$C^{*}(\mathcal{R}) \cong A$.
\end{corollary}

\begin{proof}
Since $A$ has real rank zero, the tracial state space and pairing map in the Elliott invariant of $A$ are redundant. Thus the isomorphism class of $A$ consists of those $\mathrm{C}^*$-algebras with real rank zero and isomorphic $K$-theory.  
\end{proof}

\subsection*{Acknowledgements}  The authors thank the Banff International Research Station and the organisers of the workshop “Future Targets in the Classification Program for Amenable $\mathrm{C}^*$-Algebras”, where this project was initiated. Thanks also to the Department of Mathematics and Statistics at the University of Victoria and the Department of Mathematics of the University of Colorado Boulder for research visits facilitating this collaboration. Work on the project was also facilitated by the Lorentz Center where the first and third authors attended a conference on Cuntz--Pimsner algebras in June 2018.

\end{document}